\documentclass[11pt]{amsart}

\usepackage[a4paper,left=3cm,right=3cm,top=4.5cm,bottom=4.5cm]{geometry}
\usepackage{parskip} 
\usepackage[T1]{fontenc} 

\usepackage{amsmath,amscd,amssymb,amsthm}

\usepackage{enumerate}

\usepackage{xcolor}

\usepackage[pdfencoding=auto]{hyperref} 

 \hypersetup{
     colorlinks,
     linkcolor= {blue!80!black},
     citecolor={blue!80!black},
     urlcolor={blue!80!black}}

\usepackage[nameinlink,capitalise]{cleveref}

\usepackage[english]{babel}

\usepackage{dsfont} 

\usepackage{tikz-cd} 

\usepackage{subcaption}

\newcommand{\vb}{V^{\mathrm{B}}} 
\newcommand{\gb}{\gamma^{\mathrm{B}}}

\newcommand{\C}{\mathbb{C}} 
\newcommand{\N}{\mathbb{N}}
\newcommand{\R}{\mathbb{R}}

\newcommand{\norm}[1]{\left\lVert #1 \right\rVert}
\newcommand{\normm}[1]{{\left\vert\kern-0.25ex\left\vert\kern-0.25ex\left\vert #1 \right\vert\kern-0.25ex\right\vert\kern-0.25ex\right\vert}} 

\newcommand{\br}[1]{\left\langle #1 \right\rangle}

\newcommand{\ol}[1]{\overline{#1}}

\newcommand{\1}{\mathds{1}}

\newcommand{\p}{\partial}
 
\newcommand{\dd}{\mathrm{d}}

\newcommand{\loc}{\mathrm{loc}}
\newcommand{\rad}{\mathrm{rad}}

\let\Re\relax
\DeclareMathOperator{\Re}{\mathrm{Re}}

\DeclareMathOperator{\Id}{\mathrm{Id}}

\DeclareMathOperator{\Spec}{\mathrm{Spec}}

\DeclareMathOperator{\Div}{div}

\newtheorem{proposition}{Proposition}[section]

\newtheorem{lemma}[proposition]{Lemma}
\newtheorem{theorem}[proposition]{Theorem}

\newtheorem{main-theorem}{Theorem}
\newtheorem{main-corollary}[main-theorem]{Corollary}
\newtheorem{main-algorithm}[main-theorem]{Algorithm}

\theoremstyle{definition}
\newtheorem{remark}[proposition]{Remark}

\numberwithin{equation}{section}

\newcommand{\ka}{\kappa}
\newcommand{\la}{\lambda} 
\newcommand{\La}{\Lambda} 
\newcommand{\ga}{\gamma} 
\newcommand{\om}{\omega}

\newcommand{\IB}{\mathbb{B}} 
\newcommand{\IC}{\mathbb{C}}

\newcommand{\IN}{\mathbb{N}}

\newcommand{\IR}{\mathbb{R}} 
\newcommand{\IS}{\mathbb{S}}

\newcommand{\cA}{\mathcal{A}}

\newcommand{\cC}{\mathcal{C}}

\newcommand{\cG}{\mathcal{G}}

\newcommand{\cK}{\mathcal{K}}
\newcommand{\cL}{\mathcal{L}}

\newcommand{\cN}{\mathcal{N}}
\newcommand{\cO}{\mathcal{O}}

\newcommand{\gH}{\mathfrak{H}}

\newcommand{\dL}{\Lambda} 

\newcommand{\lp}{\lambda^{\mathrm{S}}} 
\newcommand{\Lp}{\Lambda^{\mathrm{S}}}

\title[The spectrum of DtN maps for radial conductivities]{The spectrum of Dirichlet-to-Neumann maps for radial conductivities}

\author[T. Daudé]{Thierry Daudé}
\address[TD]{Université Marie et Louis Pasteur, CNRS, LmB (UMR 6623), F-25000, Besan\c con, France. CNRS – Université de Montréal CRM – CNRS}
\email{thierry.daude@univ-fcomte.fr}

\author[F. Macià]{Fabricio Macià}
\address[FM]{M$^2$ASAI. Universidad Politécnica de Madrid, ETSI Navales, Avda. de la Memoria, 4, 28040, Madrid, Spain.}
\email{fabricio.macia@upm.es}

\author[C. Meroño]{Cristóbal Meroño}
\address[CM]{M$^2$ASAI. Universidad Politécnica de Madrid, ETSI Navales, Avda. de la Memoria, 4, 28040, Madrid, Spain.}
\email{cj.merono@upm.es}

\author[F. Nicoleau]{François Nicoleau}
\address[FN]{Laboratoire de Math\'ematiques Jean Leray, UMR CNRS 6629, 2 Rue de la Houssini\`ere BP 92208, F-44322 Nantes Cedex 03.}
\email{francois.nicoleau@univ-nantes.fr}

\begin{document}

\begin{abstract}
The problem of characterizing sequences of real numbers that arise as spectra of Dirichlet-to-Neumann (DtN) maps for elliptic operators has attracted considerable attention over the past fifty years. In this article, we address this question in the simple setting of DtN maps associated with a rotation-invariant elliptic operator $\Div(\gamma\nabla\cdot)$ in the ball in Euclidean space. We show that the spectrum of such a DtN operator can be expressed as a universal term, determined solely by the boundary values of the conductivity $\gamma$, plus a sequence of Hausdorff moments of an integrable function, which we call the Born approximation of $\gamma$. We also show that this object is locally determined from the boundary by the corresponding values of the conductivity, a property that implies a local uniqueness result for the Calderón Problem in this setting. We also give a stability result: the functional mapping the Born approximation to its conductivity is Hölder stable in suitable Sobolev spaces. Finally, in order to refine the characterization of the Born approximation, we analyze its regularity properties and their dependence on the conductivity.
\end{abstract}

\maketitle

\section{Introduction}

Let $\Omega\subset\R^d$, $d\geq 2$, be a smooth bounded domain and $\ga\in \cC(\ol{\Omega})$ a \textit{conductivity}, \textit{i.e.} a real-valued continuous function that is bounded below by a strictly positive constant: $\ga(x)\geq c>0$ for all $x\in\Omega$. The Dirichlet-to-Neumann map on $\p\Omega$ associated to $\ga$ is the operator $\Lambda_\ga$ that acts on functions $f\in\cC^\infty(\p\Omega)$ as
\begin{equation}
    \Lambda_\ga f := (\ga\p_\nu u)|_{\p\Omega},
\end{equation}
where $\nu$ is the outward normal unit vector field on $\partial\Omega$, $\p_{\nu} = \nu \cdot \nabla$ and $u$ is the solution to the \textit{conductivity equation}
\begin{equation}\label{e:cond_problem} 
\begin{cases}
\Div(\gamma \nabla u)  = 0,   &\text{ in }   \Omega\subset \R^d,\\
\hfill u  = f, &\text{ on }   \partial \Omega.
\end{cases}
\end{equation}
The operator $\Lambda_\ga$ is an unbounded and positive self-adjoint operator on $L^2(\p\Omega)$ with domain $H^1(\p\Omega)$.  

Understanding the structure of the class of all DtN maps $\Lambda_\ga$ as $\ga$ varies in some class of admissible conductivities is a challenging problem that has multiple implications in various fields: deriving fine properties of their spectrum is a central problem in spectral geometry \cite{ColbSurv24}, and the problem of recovering $\ga$ from the knowledge of $\Lambda_\ga$ is known as the Calderón inverse problem \cite{FSUBook}. 

In this article, we will present strong structural properties of the set of DtN maps in the presence of rotational symmetry. We will assume that $\Omega=\IB^d$ is the unit ball in Euclidean space, and we will be interested in conductivities $\ga$ that are invariant under the action of rotations, that is \textit{radial conductivities}.  

All Dirichlet-to-Neumann operators arising from radial conductivities in $\cC(\ol{\IB^d})$
have the same invariant subspaces, namely, the spaces of spherical harmonics of some fixed degree $k\in\IN_0:=\IN\cup\{ 0\}$. Recall that these are:
\begin{equation*}
    \gH_{k,d}:=\{ P|_{\IS^{d-1}}\,:\, P \text{ homogeneous, harmonic polynomial of degree }k\text{ on }\IR^d\}.
\end{equation*}
Therefore, denoting for $k\in\IN_0$
\begin{equation*}
    \Lambda_\ga|_{\gH_{k,d}}=\lambda_k[\ga]\Id_{\gH_{k,d}},
\end{equation*}
one has that $\Lambda_\ga$ is completely determined by its spectrum
\begin{equation*}
    \Spec(\Lambda_\ga)=\{\lambda_k[\ga]\,:\,k\in\IN_0\}.
\end{equation*}
In other words, characterizing radial DtN maps amounts to characterizing their spectra. 

Since constants are always solutions to \eqref{e:cond_problem}, one always has $\lambda_0[\ga]=0$. Also, if $\ga$ is identically $1$ in $\IB^d$, one has $\la_k[\ga] = k$, for all $k\in \N_0$. In general, we will show that the rest of the eigenvalues of a radial DtN map have a very rigid structure: they can be expressed in terms of moments of radial functions. For $g\in L^1(\IB^d)$ define
\begin{equation} \label{e:moments_ball}
    \sigma_k[g] := \frac{1}{|\IS^{d-1}|} \int_{\IB^d} g(x) |x|^{2k}\, dx,\qquad k\in\IN_0,
\end{equation}
where $|\IS^{d-1}|$ denotes the measure of the unit sphere. The following holds.
\begin{main-theorem} \label{mt:conductivity}
    Let $d\ge 2$, and let $\gamma \in W^{2,p}_\rad(\IB^d)$ \footnote{Throughout this article, the subscript $\rad$ added to some function space over $\IB^d$ will mean that we are considering the subspace of its radial elements.} with $d/2<p \le  \infty$ be a conductivity. 
    Then there exists a function $\gb \in W^{2,1}_\rad(\IB^d,\IR)\subset \cC^1(\ol{\IB^d} \setminus \{0\})$ such that
    \begin{equation}\label{e:specdtn}
        \lambda_k [\gamma] = \ga|_{\IS^{d-1}} k - \frac{1}{2}\p_\nu \ga|_{\IS^{d-1}} +  \frac{1}{2}\sigma_{k}[ \Delta \gb],  \qquad\forall k\in \N ,
    \end{equation}
    and 
    \begin{equation} \label{e:traces}
             \gb|_{\IS^{d-1}} = \gamma|_{\IS^{d-1}},  \qquad \qquad    \partial_\nu \gb|_{\IS^{d-1}} = \partial_\nu \gamma|_{\IS^{d-1}} .
    \end{equation}
    In addition, $\gb$ is uniquely determined from $\Lambda_\ga$ since \eqref{e:specdtn} and \eqref{e:traces} imply that $\gb$ is characterized by any of the following equivalent conditions.
    \begin{enumerate}[i)]
        \item (Solution to a moment problem) $\gb$ is the unique solution in $L^1_\rad(\IB^d)$ to
        \begin{equation}\label{e:mompb}
            \sigma_{k}[\gb]=\frac{\lambda_{k+1}[\ga]}{2(k+1)(k+1+\nu_d)},\qquad k\in\IN_0,\qquad \nu_d:=\frac{d-2}{2}.     
        \end{equation}
        \item (Fourier transform representation) The Fourier transform of $\gb$ is given by
        \begin{equation}  \label{e:gae_formula}
        \widehat{\gb}(\xi) :=  \int_{\IB^d}e^{-i\xi\cdot x} \gb(x)\,dx = \pi^{d/2}  \sum_{k=1}^\infty  \frac{ (-1)^{k-1}}{k! \Gamma(k+d/2)}
        \left(\frac{|\xi|}{2}\right)^{2k-2}   \lambda_{k}[\gamma],
        \end{equation}
        where the series above is absolutely convergent.
    \end{enumerate}    
\end{main-theorem}
The identities \eqref{e:specdtn} and \eqref{e:gae_formula} were already obtained in a formal sense in \cite{BCMM23_n}.
Since $\Delta \gb$ is radial, one can check that the terms $\sigma_k[\Delta\gb]$ are actually the \textit{Hausdorff moments} of a function $h_\ga\in L^1([0,1])$:
\begin{equation*}
    \sigma_k[\Delta\gb] = \int_0^1 h_\ga(t)t^k \,dt, \qquad k\in\IN_0.      
\end{equation*}
Sequences of Hausdorff moments have been completely characterized by Hausdorff \cite{Haus23} (see also \cite[Chapter~3]{widd41}); therefore \Cref{mt:conductivity} provides a partial characterization of the set of radial DtN maps. Condition \eqref{e:specdtn} is strictly necessary in order to characterize the spectrum of radial DtN maps: the set of functions $\gb$ obtained as $\ga$ varies among conductivities in $W^{2,\infty}_\rad(\IB^d)$ is strictly contained in $W^{2,1}_\rad(\IB^d)$. This follows from the work of Remling \cite{Remling03} on 1-d inverse spectral theory (see \Cref{r:characterization}). 

The function $\gb$ appearing in \Cref{mt:conductivity} is called the \textit{Born approximation} of $\ga$. It is a purely spectral object, in the sense that it is completely determined by $\Spec(\Lambda_\ga)$ as \eqref{e:gae_formula} shows. However, it is local in the sense that, in any annular neighborhood of the boundary $\p\IB^d$, it is uniquely determined by the values of the conductivity in that same neighborhood.

 We define the annulus 
 \begin{equation} \label{e:us_def}
     U_s := \{ x \in \IR^d\,:\,   s<|x|\le 1 \}, \qquad 0<s<1.
 \end{equation}
\begin{main-theorem} \label{mt:conductivity_uniqueness}
    Let $d\ge 2$ and $\gamma_1 ,\gamma_2 \in  W^{2,p}_\rad(\IB^d)$ with $d/2<p \le  \infty$ be two conductivities.  Let $0<s<1$. Then
    \begin{equation*}
        \gb_1|_{U_s} = \gb_2|_{U_s}  \quad 
        \Longleftrightarrow \quad
        \gamma_1|_{U_s} = \gamma_2|_{U_s}.
    \end{equation*}
\end{main-theorem}

Structural results for the spectra of DtN maps corresponding to general, smooth, non-radial conductivities can be obtained by a detailed analysis of (pseudodifferential or Berezin) symbols of DtN maps; see, for instance,  \cite{LagAm, Roz} for the two-dimensional case or \cite{PVU} for the ball in Euclidean space.  
However, if $\ga_1, \ga_2$ are two smooth conductivities that satisfy $\ga_1|_{U_s}=\ga_2|_{U_s}$ for some $s\in(0,1)$ it is known that $\Lambda_{\ga_1} - \Lambda_{\ga_2}$ is a pseudo-differential operator whose symbol --defined modulo smoothing operators-- is zero (in the radial context, this implies that $\lambda_k[\ga_1] - \lambda_k[\ga_2] = \cO(k^{-\infty})$ as $k\to\infty$).
Therefore, these approaches are unable to determine how the interior behavior of $\ga$ affects the spectrum.
In contrast, we observe that \Cref{mt:conductivity} provides a non-trivial description of the spectrum, even in the case where the conductivity $\ga$ is constant in a neighborhood of the boundary $\partial \IB^d$.
In fact, a straightforward consequence of Theorems \ref{mt:conductivity} and \ref{mt:conductivity_uniqueness} is the following. 
\begin{main-corollary}\label{c:specrig}
Let $d\geq 2$, $d/2<p \le  \infty$, and $s\in(0,1)$. For any two conductivities $\ga_1,\ga_2\in W^{2,p}_\rad(\IB^d)$ and any $k_0\in\IN$, the following are equivalent.
\begin{enumerate}[i)]
    \item $\ga_1|_{U_s}=\ga_2|_{U_s}$.
    \item There exists $C>0$ such that
    \begin{equation}\label{e:exps}
        |\lambda_k [\gamma_1]-\lambda_k [\gamma_2]|\leq Cs^{2k},\qquad \forall k\geq k_0.
    \end{equation}
\end{enumerate}
In particular, for every conductivity $\ga\in W^{2,p}_\rad(\IB^d)$,
\begin{equation*}
    \ga|_{U_s}=\ga|_{\IS^{d-1}} \quad \iff\quad  \lambda_k[\ga] = \ga|_{\IS^{d-1}}k + \cO(s^{2k}),\qquad k\to\infty.
\end{equation*}
\end{main-corollary}
This result extends \cite[Theorem~1.4]{DKN21_stability_steklov} to $d\ge 2$ and less regular conductivities in the Euclidean setting. Due to its local nature, this result is stronger than particularizing to the radial case the standard uniqueness results for the Calderón problem. We also notice that the proof of the corollary and the previous theorems does not involve the standard approach based on Complex Geometrical Optics solutions of the equation \eqref{e:cond_problem}. An analogous result for the case of DtN maps associated to radial Schrödinger operators is given in \Cref{thm:spa} in \Cref{sec:spa}.

\Cref{mt:conductivity_uniqueness} has a quantitative counterpart, as we now show.
Define for $N>0$, $K>0$ and $1\leq p\leq\infty$ the set $\cG_{K,N}^p$ consisting of those conductivities $\gamma\in  W^{2,p}_\rad(\IB^d)$ that satisfy
\begin{equation} \label{e:ga_conditions}
    K^{-1} \le \ga(x) \le K, \quad x\in\IB^d,\qquad  \norm{\ga}_{W^{2,p}(\IB^d)} \le N. 
\end{equation}
 The Born approximation $\gb$ determines a conductivity $\ga \in \cG_{K,N}^p \subset{W^{2,1}}(\IB^d)$ in a Hölder continuous way. 
\begin{main-theorem} \label{mt:stability-gamma}
    Let $d\ge 2$, $d/2<p < \infty$, $K>1$. There exist $N_0>0$ such that, for $N>N_0$, the following assertions hold:
    \begin{enumerate}[i)]
    \item For all $0<s<1$ there exist $0<\delta<1$ and $ C_{N,K,d,p,s}>0$ such that, for every $\ga_1,\ga_2\in\cG_{K,N}^p$ satisfying $\norm{\gb_1-\gb_2}_{W^{2,1}(\IB^d)} < \delta$, one has
        \begin{equation*}
        \norm{\ga_1-\ga_2}_{W^{2,1}(U_s)}\le  C_{N,K,d,p,s}\norm{\gb_1-\gb_2}_{W^{2,1}(U_s)}^{\frac{p-1}{2p-1}}.
        \end{equation*}
    \item   For all $1<q<p$, there exist $0<\delta<1$, $0<\alpha<1$ and $ C_{N,K,d,p,q}>0$ such that for every $\ga_1,\ga_2\in\cG_{K,N}^p$ satisfying $\norm{\gb_1-\gb_2}_{W^{2,1}(\IB^d)} < \delta$ one has
        \begin{equation*}
        \norm{\ga_1-\ga_2}_{W^{2,q}(\IB^d)}  \le  C_{N,K,d,p,q} \norm{\gb_1-\gb_2}_{W^{2,1}(\IB^d)}^\alpha.
        \end{equation*}
    \end{enumerate}
    
\end{main-theorem} 
The stability estimate i) shows that the Hölder modulus of continuity of the map $\gb \mapsto \ga$ only depends on $p$ in the annuli $U_s$. This local estimate can be extended to the global estimate ii) on $\IB^d$ if one allows the Hölder modulus of continuity to depend on $N,K,p,d$. By interpolation, one can also get the norm $W^{2,q}(\IB^d)$ in the LHS of the estimate i) paying the price of a worse Hölder modulus of continuity in the RHS. 

\Cref{mt:stability-gamma} shows that $\ga$ can be determined by $\gb$ in a Hölder stable way, which is not the case for the problem of recovering a radial $\ga$ from $\Spec(\Lambda_\ga)$: this problem is ill-posed (no continuity unless the conductivities lie in a compact set \cite{Alessandrini88,Ale_Cabib_08,Faraco_Kurylev_Ruiz_14} and the modulus of continuity is generally, at best, logarithmic \cite{Mand00}).
Therefore, this shows that the ill-posedness of reconstructing a radial $\ga$ from $\Spec(\Lambda_\ga)$ is due to the linear step of solving the moment problem \eqref{e:specdtn} or \eqref{e:mompb} to reconstruct $\gb$ (the moment problem is, indeed, a notoriously ill-posed problem).

This shows that the existence of the Born approximation proved in \Cref{mt:conductivity} provides a decomposition of the process of reconstructing $\ga$ from $\Lambda_\ga$ in two steps:
\begin{itemize}
    \item The first step (obtaining $\gb$ from $\Lambda_\ga$) is linear and amounts to solving the moment problem \eqref{e:mompb}.
    \item The second step (obtaining $\ga$ from $\gb$) is non-linear but enjoys good stability and locality properties, as \Cref{mt:stability-gamma} shows.
\end{itemize}
In other words, the existence of the Born approximation and this decomposition provide a stable factorization of the Calderón problem in the radial case.

It is possible to infer more precise information on the regularity of the Born approximation. 
In general, it is always a continuous function on $\IB^d\setminus\{0\}$ that can be eventually singular at the origin. 
The nature of this singularity is related to the existence of eigenvalues and/or resonances for a certain 1-d Schrödinger operator defined in terms of $\ga$, as will be shown in \Cref{sec:resonances}.

\begin{main-theorem}  \label{mt:gb_structure}
    Let $d\ge 2$ and let $\gamma \in W^{2,p}_\rad(\IB^d)$  with $d/2<p \le  \infty$ be a conductivity. If $\nu_d := (d-2)/2$, the following assertions hold.
    \begin{enumerate}[i)]
        \item If $d>2$, there exists a finite  set of real numbers $0<\ka_1<\dots<\ka_J <\nu_d$, $J\in \N_0$, (possibly empty) and constants $c_0\in \R,\, c_1,\cdots,c_J\geq 0$ such that
        \begin{equation*}
            \gb(x) = \gb_\mathrm{reg}(x) -c_0\log|x|  + \sum_{j=1}^J \frac{c_j }{|x|^{2\kappa_j}},
        \end{equation*}
        where $\gb_\mathrm{reg} \in  \cC(\ol{\IB^d})\cap W^{1,d}(\IB^d) \cap W^{2,1}_\rad(\IB^d)$. 
        \item If $d=2$ then  $\gb \in \cC(\ol{\IB^2})\cap W^{2,1}_\rad(\IB^2)$.
        \item $\gb|_{\IB^d\setminus \{0\}} \in W^{2,p}_\loc(\IB^d\setminus \{0\})$ and for every $m\geq 2$ and $0<s<1$ one has $\gb \in \cC^m(U_s)$ if and only if $\ga \in \cC^m(U_s)$.
    \end{enumerate}
 \end{main-theorem} 
 One can  compute the Born approximation of a conductivity explicitly in some special cases, see \Cref{sec:ex}; some of these explicit examples are shown in \cref{f:ex_2} and \cref{f:ex_1}. These are chosen  to illustrate that $\gb$ can have a logarithmic singularity when $d=3$ (\cref{f:ex_2}), and to illustrate  the continuity of the Born approximation (\cref{f:ex_1}) when $d=2$ (and $d=3$ in some special cases).

The  notion of Born approximation used in this work was introduced in \cite{BCMM22} and \cite{BCMM23_n} as a formal object satisfying \eqref{e:mompb} or \eqref{e:gae_formula}, in the context of the Calderón Problem. 
This problem goes back to Calderón \cite{calderon},\footnote{This reference has been reprinted in \cite{calderon_rep}.} who considered it during the fifties, and has been the subject of intense research in the last forty years, partly because of its applications to medical imaging techniques, such as Electrical Impedance Tomography. 
We refer the reader to the recent book \cite{FSUBook} in order to obtain a wider perspective on this problem as well as an up-to-date guide to the vast literature on the subject.

By analogy with scattering problems, the name of Born approximation is motivated by the fact that $\gb$ is a linearization of the Calderón inverse problem, as \eqref{e:mompb} shows ($\gb$ depends linearly on the spectrum). Similarly as was explained in detail in \cite{Radial_Born} for the case of DtN maps associated to Schrödinger operators, or more systematically in \cite{CMMS25}, one can show, as a consequence of \Cref{mt:conductivity}, that $\gb$ satisfies the identity
\begin{equation} \label{e:frechet_main}
     \La_\ga = \dd\Psi_1(\gb),
\end{equation}
where $\Psi: \ga \longmapsto \La_\ga$ and $\dd\Psi_1$ denotes the Fréchet differential of $\Psi$ at the constant conductivity $\ga =1$.  As far as we know this is the first time in the literature that \eqref{e:frechet_main} has been shown to have an exact solution. This is, in fact, a long standing open question in the field of numerical approaches to the Calderón problem, see \cite{CMMS25} and  \cite{HaSe2010} for more details.

    \begin{figure}[t]
        \centering
        \begin{subfigure}[t]{0.46\textwidth}
            \centering 
            \includegraphics[width=0.9\textwidth]{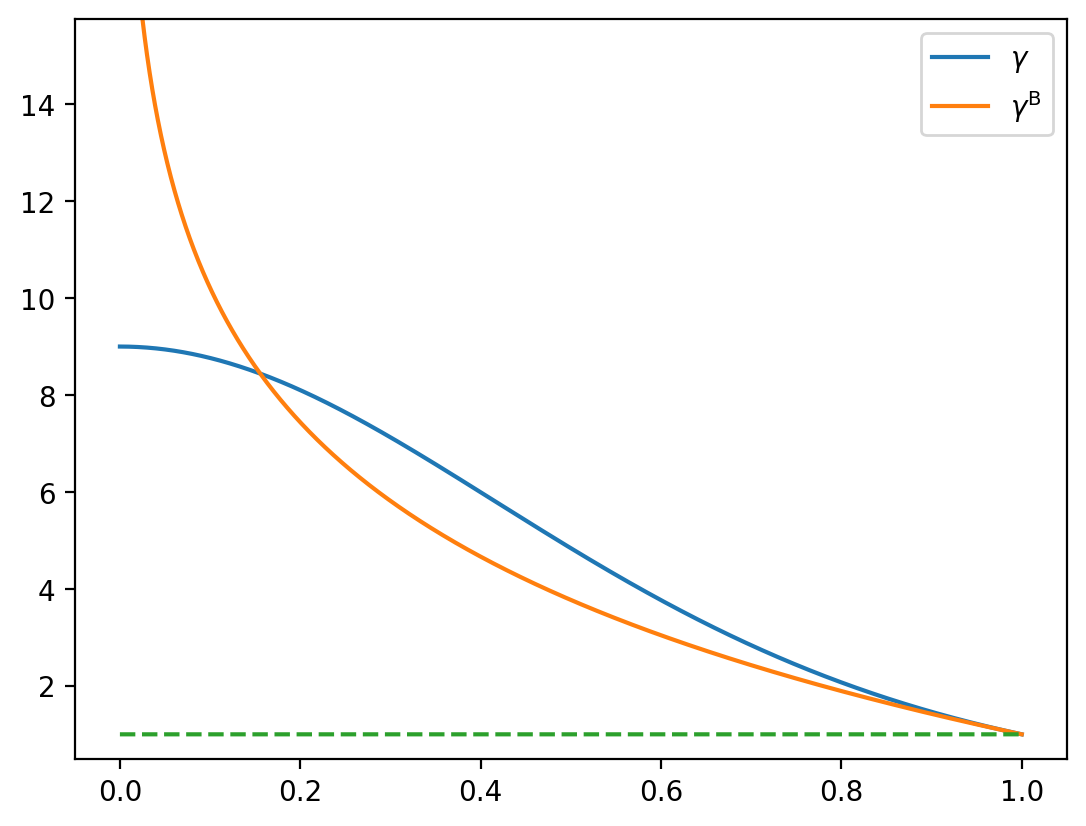}
        \end{subfigure}
        \caption{Plot of the radial profile of $\ga_{d,\mu,\nu}$ (blue) and $\gb_{d,\mu,\nu}$  (orange) defined in \Cref{sec:ex} for $d=3,\;\mu=1,\;\nu=0$. One can observe the presence of a logarithmic singularity in the Born approximation induced by a resonance at zero.} 
        \label{f:ex_2}
\end{figure}

The approach based on the Born approximation has also been successfully applied in the radial case to the Schrödinger operator version of the Calderón problem (the so-called Gelf'and-Calderón problem) at zero energy \cite{Radial_Born}, and at fixed energy \cite{MMS-M25}, and has been formulated in more general settings in \cite{MaMe24}. In particular, analogous results to Theorems 1-4 can be found in \cite{Radial_Born,MMS-M25} for the respective problems. Linearization at other conductivities besides the constant conductivities and existence of the corresponding Born approximation has been addressed in \cite{CMMS25}, as well as some partial results for the non-radial case in two dimensions.
We remark that, using the reconstruction algorithms in \cite{Radial_Born}, one could transform \Cref{c:specrig} into an explicit reconstruction algorithm to determine $\ga$ from $\La_\ga$.

The proofs of the previous theorems use the notion of $A$-amplitude introduced by Simon in \cite{IST1} in the context of inverse spectral theory for half-line Schrödinger operators.
This notion of $A$-amplitude, further developed in \cite{IST2,IST3}, is also a fundamental tool in \cite{Radial_Born,MMS-M25}. 
In particular, to prove \Cref{mt:stability-gamma} we use stability results for the problem of reconstructing a Schrödinger potential from its $A$-amplitude developed in \cite[Section 5]{Radial_Born}. 
The connection between the $A$-amplitude and $\gb$ is not superficial: it has been shown in \cite{MaMe24} that the $A$-amplitude is a Born approximation for the inverse problem of reconstructing a Schrödinger potential from its Weyl-Titchmarsh function. 
The $A$-amplitude also plays an important role in the closely related problem of studying spectra of DtN maps in warped-product manifolds with boundary, see \cite{DKN20,DKN21_stability_steklov,DKN23}.

\subsection*{Structure of the article. }\Cref{mt:conductivity}, \Cref{mt:conductivity_uniqueness}, \Cref{c:specrig} and \Cref{mt:gb_structure} are proved in \Cref{sec:born_exists}, whereas \Cref{mt:stability-gamma} is proved in \Cref{sec:stability}. \Cref{sec:resonances} contains a detailed analysis of the structure of singularities of the Born approximation, both in the Schrödinger operator and conductivity cases, that are used in the proofs of \Cref{mt:conductivity} and \Cref{mt:gb_structure}; these results are of independent interest.

\subsection*{Acknowledgments}
F.M. and C.M. have been supported by grants PID2021-124195NB-C31 and PID2024-158664NB-C21 from Agencia Estatal de Investigación (Spain). F. N. is supported by the GDR Dynqua.

    \begin{figure}[h]
        \centering
        \begin{subfigure}[t]{0.46\textwidth}
            \centering
            \includegraphics[width=0.9\textwidth]{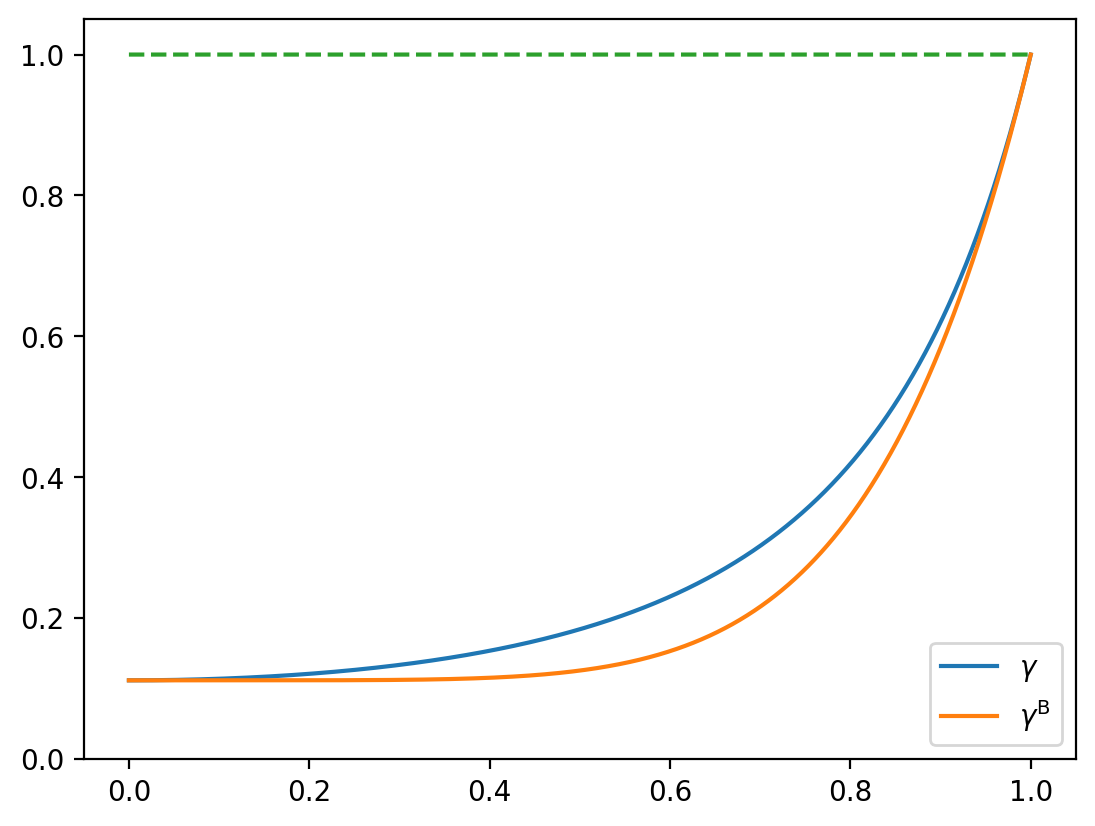}
            \caption{$d=2,\;\mu=1,\;\nu=3$.}  
        \end{subfigure}
        \hfill
        \begin{subfigure}[t]{0.46\textwidth} 
            \centering 
           \includegraphics[width=0.9\textwidth]{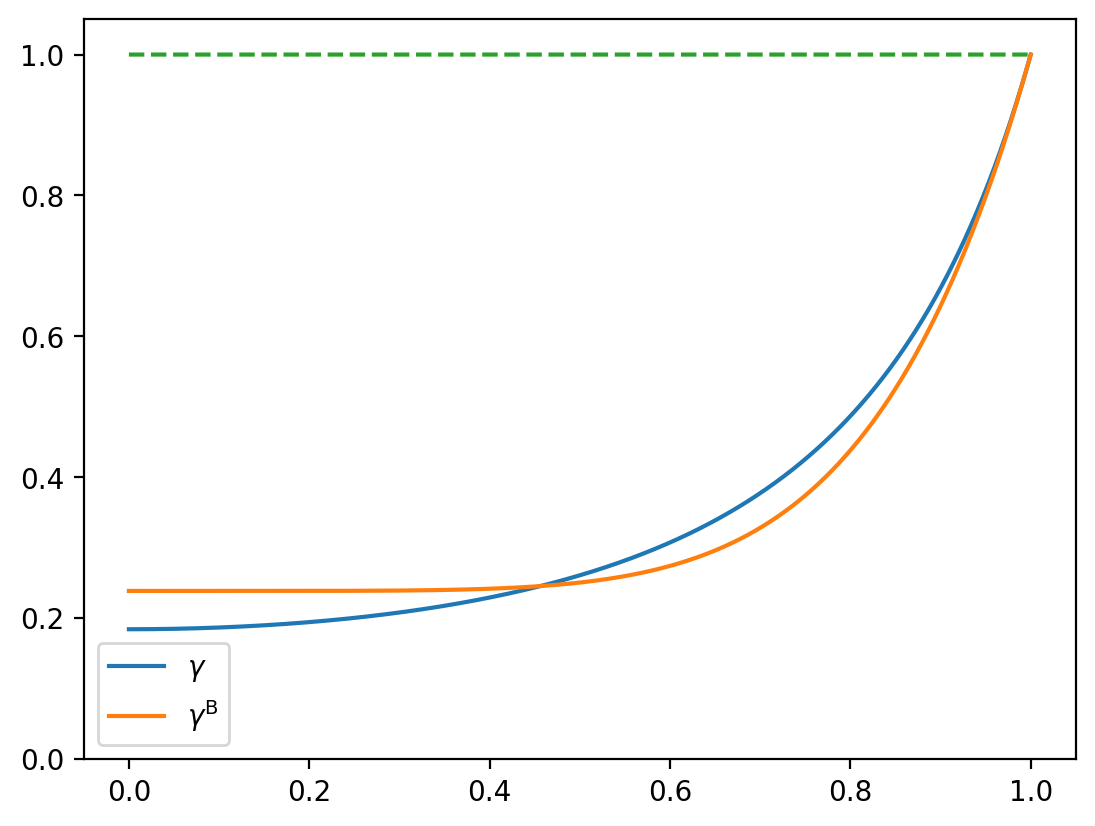}
            \caption{$d=3,\;\mu=1,\;\nu=3$.}  
        \end{subfigure}
        \vskip\baselineskip
        \begin{subfigure}[t]{0.46\textwidth}
            \centering 
            \includegraphics[width=0.9\textwidth]{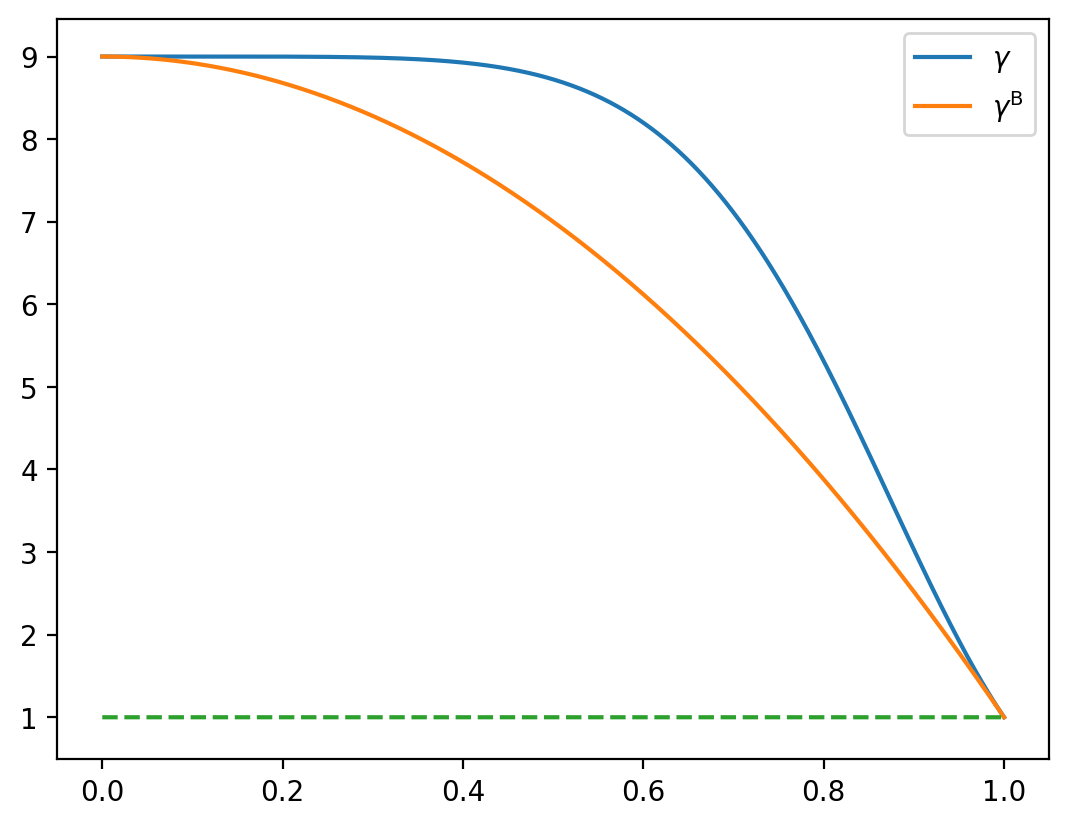}
            \caption{$d=2,\;\mu=3,\;\nu=1$.}  
        \end{subfigure}
        \hfill
        \begin{subfigure}[t]{0.46\textwidth}
            \centering 
            \includegraphics[width=0.9\textwidth]{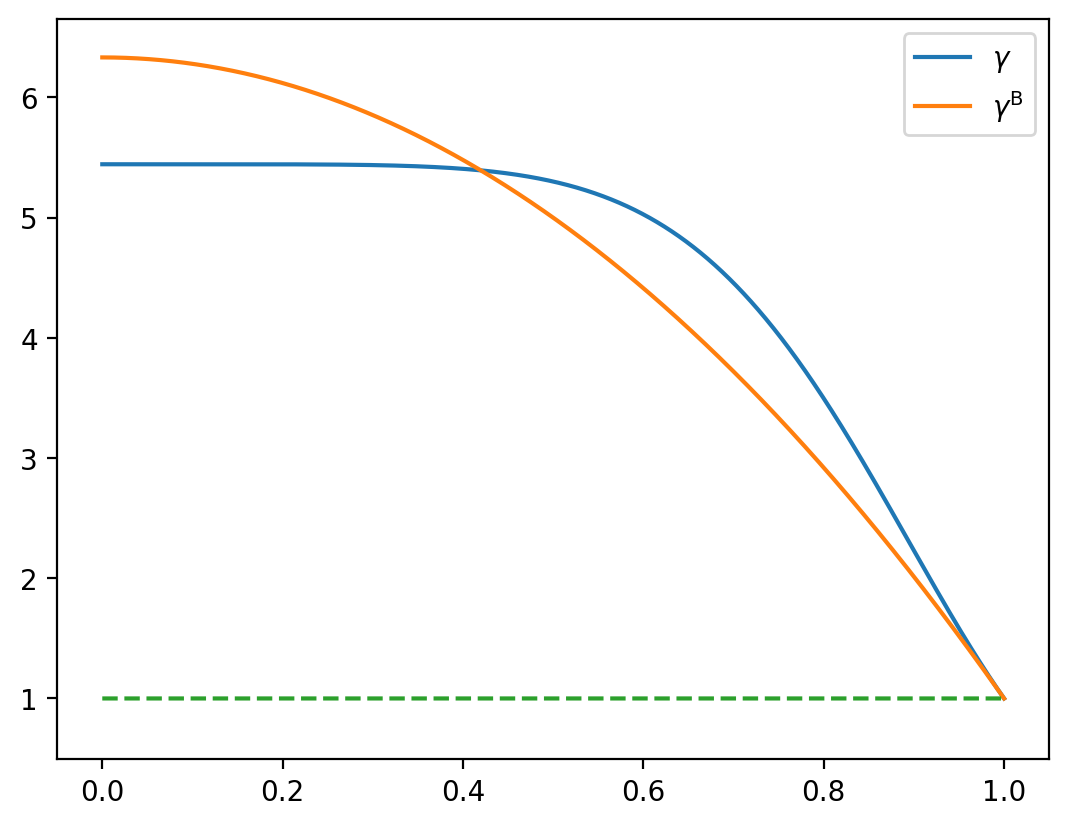}
            \caption{$d=3,\;\mu=3,\;\nu=1$.}  
        \end{subfigure}
        \vskip\baselineskip
        \caption{Plots of the radial profiles of $\ga_{d,\mu,\nu}$ (blue) and $\gb_{d,\mu,\nu}$  (orange) defined in \Cref{sec:ex}.}
        \label{f:ex_1}
\end{figure}

\section{Existence of the Born approximation}\label{sec:born_exists} 

\subsection{The Born approximation for DtN maps of Schrödinger operators} \label{sec:DtNp}\

It is well-known that, provided $\ga$ has enough regularity, the study of $\Lambda_\gamma$ can be reduced to that of a DtN map associated with a certain Schrödinger operator. If $V\in L^p(\IB^d,\IR)$ with $p>d/2$ is such that zero is not a Dirichlet eigenvalue of $-\Delta+V$ on $L^2(\IB^d)$, then for every $f\in H^{1/2}(\IS^{d-1})$,  the problem 
 \begin{equation}   \label{e:bv_prob_S}
\begin{cases}
-\Delta v + V v  &= 0, \qquad    \text{on }   \IB^d,\\
\hfill v|_{\IS^{d-1}}  &= f,
\end{cases}
\end{equation}
admits a unique solution $v\in H^1(\IB^d)$. The DtN map is defined by the identity $\Lp_V f = \partial_\nu v|_{\IS^{d-1}}$, where $v$ is the solution of \eqref{e:bv_prob_S}. If the regularity of $V$ or $f$ is not enough to define $\Lp_V f$, one uses the standard weak definition:
\begin{equation*}
 \br{g, \Lp_{V}f}_{H^{1/2}(\IS^{d-1}) \times H^{-1/2}(\IS^{d-1})} :=  \int_{\IB^d} \left( \nabla u\cdot \nabla v +  V u   v \right)\, dx,
\end{equation*}
where $v$ solves \eqref{e:bv_prob_S}  and $u$ is any $H^{1}(\IB^d)$ function with trace $g$. 

If $\gamma\in W^{2,p}(\IB^d,\IR_+)$ with $p>d/2$ is a conductivity and $u$ solves \eqref{e:cond_problem}, then $v=\sqrt{\gamma}u$ satisfies \eqref{e:bv_prob_S} with boundary datum $\sqrt{\gamma}f$,
where 
\begin{equation} \label{e:V_to_gamma}
V:=\frac{\Delta \sqrt{\gamma}}{\sqrt{\gamma}}\in L^p(\IB^d,\R).
\end{equation}
Moreover, if $V$ arises in this way from a conductivity, then zero is not a Dirichlet eigenvalue of $-\Delta+V$ on $L^2(\IB^d)$, and one can check that the following identity holds
\begin{equation} \label{e:DNmap}  
\Lp_V f =\gamma^{-1/2}  \dL_\gamma (\gamma^{-1/2}f) + \frac{1}{2} \gamma^{-1} (\partial_\nu\gamma) f.
\end{equation}
From this identity, it is clear that the Calderón problem for conductivities is equivalent to the one with Schrödinger operators, provided $\gamma$ and $\partial_\nu \gamma$ are known at the boundary $\IS^{d-1}$. 

As was the case with conductivities, if $V\in L^p_{\rad}(\IB^d,\IR)$, then $\Lp_V$ is a self-adjoint unbounded operator on $L^2(\IS^{d-1})$ that is invariant by rotations (since the equation \eqref{e:bv_prob_S} is). Therefore, it is diagonalizable in the basis of spherical harmonics. We denote:
\begin{equation*}
    \Lp_V|_{\gH_{k,d}}=\lp_k[V] \Id_{\gH_{k,d}},\qquad \forall k\in\N_0.
\end{equation*} 
Let $\1 \in L^\infty(\IB^d)$ be the conductivity that is identically 1 in $\IB^d$. When $\ga =\1$, the associated potential is $V=0$
 and hence $\Lp_0 = \La_\1$ and  $\lp_k[0] = \lambda_k[\1] = k$ for every $k\in\N_0$.

The following result improves the conclusion of \cite[Theorem 1]{Radial_Born} when the potential arises from a conductivity via identity \eqref{e:V_to_gamma}. For shortness, define
\begin{equation*}
    \norm{f}_{d,\alpha}:=\int_{\IB^d} |f(x)| |x|^{2-d}  (1-\log|x|)^\alpha \, dx <\infty, \qquad  \alpha \ge 0,\,  d\ge 2.
\end{equation*}
\begin{proposition}  \label{p:Vb_space}
    Let $d\ge 2$ and let $\gamma \in W^{2,p}_\rad(\IB^d)$  with $d/2<p\le \infty$ be a conductivity. Let $V$ be given by \eqref{e:V_to_gamma} and $\Lp_V $ by \eqref{e:DNmap}. Then $V \in L^{p}_\rad(\IB^d,\R)$ and there exists a unique $\vb \in L^1_\rad(\IB^d,\R)$ such that
    \begin{equation} \label{e:moment_prob_pot}
        \sigma_k[\vb] = \lp_k[V]-   k,     \qquad \text{for all } \, k \in \N_0.
    \end{equation}
    In addition:
    \begin{enumerate}[i)]
        \item If $d>2$, there exists a finite (possibly empty) set of real numbers  $0<\ka_1<\dots<\ka_J <\nu_d$ with $J\in \N_0$ and some constants $a_0\in \R, \, a_1,\dots,a_J\le 0$ such that
        \begin{equation}\label{e:singsvb}
            \vb(x) = \vb_\mathrm{reg}(x) + \frac{a_0}{|x|^2} + \sum_{j=1}^J \frac{a_j }{|x|^{2\kappa_j+2}} ,
        \end{equation}
        where $\vb_\mathrm{reg}$ satisfies $\norm{\vb_\mathrm{reg}}_{d,\alpha}<\infty$ for all $\alpha>1$. 
        \item If $d=2$ then $\norm{\vb}_{2,\alpha}<\infty$ for all $\alpha>1$.
        \item $\vb|_{\IB^d\setminus \{0\}} \in L^p_\loc(\IB^d\setminus \{0\})$. In fact, $\vb - V$ is continuous on $\ol{\IB^d}\setminus\{0\}$ and vanishes on $\partial \IB^d$. Also, $V \in \cC^m(U_s)$ iff $\vb \in \cC^m(U_s)$, for all $m\in\IN_0$ and for all $0<s<1$.
        \item $\vb \in L^1(\IB^d)$.
    \end{enumerate}
 \end{proposition}
Following  \cite{Radial_Born}, we say that $\vb$ is the Born approximation of the potential $V$.
\Cref{p:Vb_space}  is proved in \Cref{sec:specconduc}. It relies on a fine analysis of the singularities of the Born approximation and constitutes an important ingredient of the proof of \Cref{mt:conductivity} presented below.

\subsection{The Born approximation \texorpdfstring{for $\Lambda_\ga$}{}} \

The proof of \Cref{mt:conductivity} follows from \Cref{p:Vb_space} and the following technical properties of radial functions.
\begin{lemma}  \label{l:w_rad}
    Let $d\ge 2$.  The following assertions hold:
    \begin{enumerate}[i)]
        \item If $u \in W^{2,1}_\rad(\IB^d)$ then $u \in \cC^1(\ol{\IB^d} \setminus \{0\})$. Moreover, if $u= u_0(|\cdot|)$,  the following quantities are well-defined
        \begin{equation} \label{e:ab_notation}
        a(u) := u_0(1) , \qquad b(u) := \frac{1}{2}\partial_r u_0(1),
        \end{equation}
        and satisfy $|a(u)| \le C\norm{u}_{W^{2,1}(\IB^d)}$, $|b(u)| \le C\norm{u}_{W^{2,1}(\IB^d)}$ for some $C>0$.
        \item \label{item:last} For all $u \in W^{2,1}_\rad(\IB^d)$
        \begin{align}  
        \sigma_0[\Delta u] & =2b(u),\label{e:moments_0}\\
        \sigma_{k}[\Delta u] &=   2b(u) -2k a(u) + 2k(2k+d-2) \sigma_{k-1}[u], \qquad \forall k\in \N.\label{e:moments_derivative}
        \end{align}
    \item For all $u \in L^1_\rad(\IB^d)$
    \begin{equation}  \label{e:H_fourier_0}
    \widehat{u}(\xi) 
    =\pi^{d/2}  \sum_{k=1}^\infty  \frac{ (-1)^{k-1}}{k! \Gamma(k+d/2)}
    \left(\frac{|\xi|}{2}\right)^{2k-2}   k(2k+d-2) \sigma_{k-1}[u].
    \end{equation}
    \end{enumerate}
\end{lemma}
\begin{proof}
    Statement i) is immediate using that $\partial_r^2 u_0(r)$ belongs to $L^1(\varepsilon,1)$ for all $\varepsilon>0$. The boundedness of the trace  functionals $a$ and $b$ follows immediately by the standard trace theorem with exponent $p=1$.   We now prove the remaining items. By integration by parts:
    \begin{align*}
     \nonumber \sigma_k[\Delta u] =  \int_0^1 &\partial_r (r^{d-1} \partial_r u_0(r)) r^{2k} \, dr =\partial_r u_0(1)-2k\int_0^1   \partial_r u_0(r) r^{2k+d-2} \, dr \\
      &= \partial_r u_0(1) -2k u_0(1) + 2k(2k+d-2)\int_0^1  u_0(r) r^{2k +d-3} \, dr \\
       &= \partial_r u_0(1) -2k u_0(1) + 2k(2k+d-2) \sigma_{k-1}[u].
    \end{align*}
    Thus, for any $u\in W^{2,1}_\rad(\IB^d)$ we have 
    \begin{equation*}  
        \sigma_{k}[\Delta u] =   \partial_r u_0(1) -2k u_0(1) + 2k(2k+d-2) \sigma_{k-1}[u],
    \end{equation*}
    which proves \eqref{e:moments_derivative}. The previous derivation also holds with $k=0$, in which case one gets $ \sigma_{0}[\Delta u]  = \partial_r u_0(1)$. This concludes the proof of ii). Formula \eqref{e:H_fourier_0} follows from \cite[Equation 1.20]{BCMM22}, which states that
    \begin{equation}\label{e:ft_moments}
    \widehat{u}(\xi) =   2 \pi^{d/2}  \sum_{k=0}^\infty  \frac{ (-1)^k}{k! \Gamma(k+d/2)}
    \left(\frac{|\xi|}{2}\right)^{2k}   \sigma_k[u],
    \end{equation}
    and straightforward algebraic properties of the gamma function.
\end{proof}

Define the following operator acting on $L^1_\rad(\IB^d)$. 
\begin{equation} \label{e:Td_def}
    T_d f(x) := \int_{|x|}^1 \frac{1}{t^{d-1}} \int_0^t   f_0(s) s^{d-1}\, ds dt,\qquad f\in L^1_\rad(\IB^d),\; f = f_0(|\cdot|).
\end{equation}
One readily checks that $T_d f\in \cC^1(\ol{\IB^d}\setminus\{0\})$ is radial, $T_d f|_{\IS^{d-1}} =0$, and that $-\Delta T_d f= f$ a.e. in $\IB^d$.
\begin{lemma} \label{l:regularity_poisson}

    Let  $f\in L^1_\rad(\IB^d)$ and $u:=T_d(f)$.
    
    \begin{enumerate}[i)]
        \item If $f\in L^p_\rad(\IB^d)$ for some $1<p< \infty$, then $u \in W^{2,p}_\rad(\IB^d)$.
        \item   If $f$ satisfies  for some $\alpha > 1$  that 
        \begin{equation*}
       \norm{f}_{d,\alpha} = \int_{\IB^d} |f(x)| |x|^{2-d}  (1-\log|x|)^\alpha \, dx <\infty, 
        \end{equation*} 
        then $u \in \cC(\ol{\IB^d})$, $\nabla u \in L^d(\IB^d)$, and $D^2 u \in L^1(\IB^d)$. 
    \end{enumerate}
    
\end{lemma}
\begin{proof}
    i) Since $-\Delta u = f$, $ f\in L^p(\IB^d)$ and $u$ vanishes on $\IS^{d-1}$, $u\in W^{2,p}(\IB^d)$ by the standard Calderón-Zygmund estimates for the Laplacian.
    
    ii) Denote by $F$ and $f_0$, respectively, the radial profiles of $u = T_d(f)$ and $f$. By \eqref{e:Td_def} we have
    \begin{equation*}
        |F'(r)| \le \frac{1}{r^{d-1}}\int_0^r   |f_0(s)| s^{d-1} \, ds   \le \frac{1}{r}\int_0^r   |f_0(s)| s \, ds .
    \end{equation*}
    Then, for all $\alpha >1$, using that $(1-\log r) \le (1-\log s) $ for $s<r$, we have the estimate
    \begin{equation*}
        |F'(r)|  \le \frac{1}{r(1-\log r)^\alpha} \int_0^r   |f_0(s)| s (1-\log s)^\alpha \, ds \le     \frac{c_d\norm{f}_{d,\alpha} }{r(1-\log r)^\alpha},
    \end{equation*}
    where $c_d:= |\IS^{d-1}|^{-1}$. This implies that $F(0)$ is well defined and finite, since, under our hypothesis on $\alpha$,
    \begin{equation*}
        |F(0)|  \le  \int_{0}^1 |F'(r)| \, dr \le \int_{0}^1 \frac{c_d \norm{f}_{d,\alpha} }{r(1-\log r)^\alpha} \, dr =\int_0^{\infty} \frac{c_d \norm{f}_{d,\alpha} }{(1+t)^\alpha}dt <\infty.
    \end{equation*}
    The integrability of $F'$ in $[0,1]$ implies that $F$ is continuous on $[0,1]$, and hence, $u \in \cC(\ol{\IB^d})$.
    Since $\alpha>1$ then $\alpha d>1$ and one finds that the gradient satisfies
    \begin{equation*}
        \norm{\nabla u}_{L^d(\IB^d)}^d \le |\IS^{d-1}|\int_0^1 |F'(r)|^d r^{d-1} \, dr\le  \int_0^1 \frac{c_d^{d-1}\norm{f}_{d,\alpha}^d }{r(1-\log r)^{ d \alpha }} \, dr <\infty.
    \end{equation*}
    One can similarly bound the Hessian, using $ F''(r) = -(d-1)r^{-1}F'(r) -f_0(r)$,
    \begin{equation*}
        |\p_{x_i x_j}^2 u(x)| \le |F''(|x|)| + 2|F'(|x|) ||x|^{-1} \le (d+1)|F'(|x|) ||x|^{-1} + |f(x)|.
    \end{equation*}
    and,
    \begin{align*}
        \norm{D^2 u}_{L^1(\IB^d)}
        &\le  (d+1)\int_{\IB^d} |F'(|x|) | |x|^{-1}\, dx  + \norm{f}_{L^1(\IB^d)} \\
        &\le (d+1) \norm{f}_{d,\alpha}\int_{0}^1  \frac{r^{d-3}}{(1-\log r)^{  \alpha }} \, dr  + \norm{f}_{L^1(\IB^d)}.
    \end{align*}
    The first integral converges since $\alpha>1$, and, by hypothesis, $\norm{f}_{L^1(\IB^d)}<\infty$; therefore $\norm{D^2 u}_{L^1(\IB^d)}<\infty$ as we wanted to prove.
\end{proof}

\begin{proof}[Proof of \Cref{mt:conductivity}]

Let $V$ be given by \eqref{e:V_to_gamma} and restrict the operator identity \eqref{e:DNmap} to functions in $\gH_{k,d}$ to obtain
\begin{equation} \label{e:ga_eigenvalue}
 \lp_k[V] -k  = \frac{1}{a(\gamma)} \left(   \lambda_k[\gamma] +  b(\gamma) -a(\gamma) k  \right)  , \qquad \forall \, k\in \N_0 ,
\end{equation}
where the  notation in \eqref{e:ab_notation} has been used.
On the other hand, since $\ga \in W^{2,p}_\rad(\IB^d)$, we have $V\in L^{p}_\rad(\IB^d,\IR)$. Therefore, \Cref{p:Vb_space} applies and there exists $\vb \in L^1_\rad(\IB^d,\IR)$ such that
\begin{equation*} 
\lp_k[V] -k  = \sigma_k[\vb] , \qquad \forall \, k\in \N_0 .
\end{equation*}
Inserting this in \eqref{e:ga_eigenvalue} yields
\begin{equation}   \label{e:gamma_eigenvalue}
  \lambda_k[\gamma]  =  a(\gamma) k  - b(\gamma) + a(\gamma) \sigma_k[\vb] , \qquad \forall \, k\in \N_0    .
\end{equation}
We now choose $\gb$ so that $\Delta \gb =  2a(\gamma) \vb$. In order to do so, define
\begin{equation}\label{e:gb_def_thm1}
    \gb(x) :=  a(\ga)-2a(\ga)T_d(\vb),
\end{equation}
where $T_d$ was defined in \eqref{e:Td_def}. Then $\gb \in \cC^1(\ol{\IB^d}\setminus\{0\})$. 

By \Cref{p:Vb_space}, $\vb = \vb_\mathrm{sing} + \vb_\mathrm{reg}$ where
$\norm{\vb_\mathrm{reg}}_{d,\alpha}<\infty$ for all $\alpha>1$, and either $\vb_\mathrm{sing}=0$, if $d=2$, or $\vb_\mathrm{sing} \in L^q_\rad(\IB^d)$ for every $1\leq q<d/s_V$, with $s_V=2\max_{j=0,...,J}\{1,\ka_j+1\}< d$ when $d\geq3 $. Therefore, by \Cref{l:regularity_poisson}, one gets that $\gb \in W^{2,1}_\rad(\IB^d)$.
Also, $\gb$ solves the Poisson problem
\begin{equation}  \label{e:from_comp}
\Delta \gb =  2a(\gamma) \vb , \qquad a(\gb) = a(\ga).
\end{equation}
This proves \eqref{e:specdtn} by direct substitution in \eqref{e:gamma_eigenvalue}.  

To prove \eqref{e:traces}, it remains to show that $b(\ga) = b(\gb)$. Evaluating  \eqref{e:gamma_eigenvalue} in $k=0$ and using that $\lambda_0[\gamma]=0$, we get that $  b(\ga) = a(\ga) \sigma_0[\vb]$. Using that $\Delta \gb =  2a(\gamma) \vb$ and \eqref{e:moments_0}, we obtain
\begin{equation*}
     b(\ga) = a(\ga) \sigma_0[\vb] = \frac{1}{2}\sigma_0[\Delta \gb]=b(\gb).
\end{equation*}
Hence the identities \eqref{e:specdtn} and \eqref{e:traces} hold. This proves that the Born approximation exists.

Applying \eqref{e:moments_derivative} to $\gb$ and inserting the identity \eqref{e:traces} gives
\begin{equation*}
\frac{1}{2}\sigma_k[\Delta \gb] = b(\gamma) -k a(\gamma) + k(2k+d-2) \sigma_{k-1}[\gb], \quad \forall k\in \N.
\end{equation*}
Combining this identity with \eqref{e:specdtn} we deduce that
\begin{equation*}  
   \lambda_k [\gamma] =   k\left(2k +d-2 \right) \sigma_{k-1}[ \gb],  \qquad\forall k\in \N,
\end{equation*}
or, equivalently, that  
$\gb$ is a solution of the moment problem \eqref{e:mompb}. 
This problem has  uniqueness of solutions in $L^1_\rad(\IB^d)$, as the reconstruction formula  \eqref{e:H_fourier_0} implies (in fact, from  \cite[Lemma 6.2]{Radial_Born} it follows that uniqueness holds even for compactly supported, radial distributions). 
Identity \eqref{e:gae_formula} follows directly from \eqref{e:H_fourier_0}.
\end{proof}

\subsection{Proof of the local uniqueness results}

\begin{proof}[Proof of \Cref{mt:conductivity_uniqueness}]
    Assume first that $\gb_1 = \gb_2$ on  $U_s$. Then, in particular \eqref{e:traces} implies
    \begin{equation} \label{e:traces_unique}
             \gamma_1|_{\IS^{d-1}} = \gamma_2|_{\IS^{d-1}},  \qquad \qquad    \partial_\nu \gamma_1|_{\IS^{d-1}} = \partial_\nu \ga_2|_{\IS^{d-1}} .
    \end{equation}
    Also, we get $\vb_1 = \vb_2$ on  $U_s$ by \eqref{e:from_comp}, which implies $V_1 = V_2$ on  $U_s$ by \cite[Theorem 2]{Radial_Born}.
    Then, if $u:= \sqrt{\ga_1} -\sqrt{\ga_2}$, in $U_s$ we have by \eqref{e:V_to_gamma} and \eqref{e:traces_unique} that
    \begin{equation*}  
    \begin{cases}
        -\Delta u + V u = \,0  \quad \text{ on }   U_s,\\
        u|_{\IS^{d-1}}  =  \, 0, \quad \partial_\nu u|_{\IS^{d-1}} =0,   \\ 
    \end{cases}
    \end{equation*}
    where $V := V_1 = V_2$ on $U_s$. Since $V \in L^p (U_s)$ with $ p > \frac{d}{2}$, by unique continuation, we deduce that $u=0$ on $U_s$ and therefore $\gamma_1 = \gamma_2$ on $U_s$.

    If we now assume $\gamma_1 = \gamma_2$ on $U_s$, then \eqref{e:V_to_gamma} implies $V_1 = V_2$ on  $U_s$, and hence one gets  $\vb_1 = \vb_2$ on  $U_s$ by \cite[Theorem 2]{Radial_Born}. Finally, \eqref{e:from_comp} and \eqref{e:traces_unique} (which again holds by \eqref{e:traces}) imply that $\gb_1 = \gb_2$ on  $U_s$ by unique continuation.
\end{proof}

Before proving \Cref{c:specrig}, we need to prove a local uniqueness result for the Hausdorff Moment Problem.
\begin{lemma}\label{l:mom}
For $F\in L^1_\rad(\IB^d)$, $s\in(0,1)$, and $k_0\in\IN$, the following holds:
\begin{equation*}
    F|_{U_s}=0 \;\iff\; \exists C>0,\quad |\sigma_k[F]|\leq Cs^{2k},\quad k\geq k_0.
\end{equation*}
\end{lemma}
\begin{proof}
    If $F|_{U_s}=0 $ it is straightforward to verify that $|\sigma_k[F]|\leq Cs^{2k}$ for all $k \ge 0$. It remains to check the other implication.  If $F=f(|\cdot|)$ one has $\sigma_k[F]=\cL(F)(k+d/2)$ where
    \begin{align*}
        \cL(F)(z):= \int_0^\infty f(e^{-t})e^{-z 2t}dt=\frac{1}{2}\int_0^\infty f(e^{-t/2})e^{-z t}dt, 
    \end{align*}
    and there exists $C_s>0$ such that
    \begin{equation}\label{e:gbdd}
        \left|\int_0^{2|\log s|} f(e^{-t/2})e^{-z t}dt\right| \leq |\cL(F)(z)|+C_s\norm{F}_{L^1(\IB^d)}e^{-z2|\log s|}, \qquad z>d/2.
    \end{equation}
    Define
    \begin{equation*}
        G(z):=e^{z2|\log s|}\int_0^{2|\log s|} f(e^{-t/2})e^{-z t}dt.
    \end{equation*}
    By hypothesis, $|\sigma_k[F]|\leq Cs^{2k}$ for $k\geq k_0$, which implies that 
    \[
        e^{(2k+d)|\log s|}|\cL(F)(k+d/2)|\leq Cs^{-d}, \qquad \forall k\ge k_0.
    \]
     Hence, in view of \eqref{e:gbdd}, $G$  satisfies
    \begin{equation*}
       |G(k+d/2)|\leq  Cs^{-d} + C_s  \norm{F}_{L^1(\IB^d)} \qquad \forall k\ge k_0.
    \end{equation*}
    Thus, $G$ is an entire function that satisfies $\sup_{k\in \N} |G(k+d/2)| \le \tilde{C}$ for some $\tilde{C}>0$ and
    \begin{equation*}
        |G(z)|\leq \left(\int_0^{2|\log s|}|f(e^{-t/2})|dt\right)e^{2|\log s|(\Re z)_+}.
    \end{equation*}
    One can apply a theorem of Duffin and Schaeffer \cite[Theorem~10.5.1]{Boas} to conclude that $M>0$ exists such that $|G(z)|\leq M$ when $z>0$. This implies that
    \begin{equation*}
        \int_0^{2|\log s|} f(e^{-t/2})e^{-z t}dt=\cO(e^{-z2|\log s|}), \qquad z\to\infty.
    \end{equation*}
    Using the local injectivity of the Laplace transform, see \cite[Lemma~A.2.1]{IST1}, we infer that $f(e^{-t/2})=0$ for a.e. $t\in (0,2|\log s|)$, hence $f(r) = 0$ for a.e. $r\in (s,1)$, and the conclusion follows.
\end{proof}
\begin{proof}[Proof of \Cref{c:specrig}]
This result is a straightforward consequence of \Cref{mt:conductivity}, \Cref{mt:conductivity_uniqueness} and \Cref{l:mom}.
To see that two conductivities satisfying \eqref{e:exps} must necessarily coincide on $U_s$, start by noticing that \eqref{e:specdtn} and \eqref{e:exps} imply that $\ga_1|_{\IS^{d-1}}=\ga_2|_{\IS^{d-1}}$ and $\p_\nu\ga_1|_{\IS^{d-1}}=\p_\nu\ga_2|_{\IS^{d-1}}$. 
\Cref{l:mom} ensures then that $(\Delta \gb_1)|_{U_s}=(\Delta \gb_2)|_{U_s}$; therefore, by \eqref{e:traces}, $w:=\gb_1-\gb_2$ solves $\Delta w=0$ on $U_s$ with $w|_{\IS^{d-1}}=\p_\nu w|_{\IS^{d-1}}=0$. 
This implies $w|_{U_s}=0$; the result follows now from \Cref{mt:conductivity_uniqueness}.
\end{proof}

\begin{proof}[Proof of \Cref{mt:gb_structure}]
 By \eqref{e:gb_def_thm1} we have that $\gb = a(\ga) -2a(\ga)T_d(\vb)$. Using this, the case ii) ($d=2$) and the regularity of $\gb_\mathrm{reg}$ in i) ($d \ge 3$) are immediate by \Cref{p:Vb_space} and \Cref{l:regularity_poisson}~ii). 
 
 To finish the proof of i)  notice that
 $T_d(|\cdot|^{-2\ka -2})  = \frac{|x|^{-2\ka}-1}{2\ka(d-2-2\ka)}$ for all $0<\ka<\nu_d$, and $T_d(|\cdot|^{-2})  = -(d-2)^{-1}\log|x|$, $d\ge 3$. 
 This yields using \Cref{p:Vb_space}~i) the representation in $d\ge 3$ with $c_j = - a_j\frac{1}{2\ka_j(d-2-2\ka_j)}$ for $j=1,\dots,J$. Since $a_j \le 0$ for all $j\ge 1$, then $c_j \ge 0$  for all $j\ge 1$. 
 Finally iii) follows immediately from \Cref{p:Vb_space}~iii) using that $\gb$ and $\sqrt{\ga}$ are, respectively, solutions of the elliptic equations $\Delta \gb = 2a(\ga)\vb$ and $\Delta \sqrt{\ga} = V \sqrt{\ga}$.
\end{proof}

\section{The Born approximation at the origin: singularity and resonances}  \label{sec:resonances}

\subsection{Weyl-Titchmarsh functions and DtN maps}
Write
\[
\nu_d := \frac{d-2}{2},\qquad \IC_+ = \{z\in \IC: \Re(z) >0\}.
\]
In the presence of radial symmetry, it is possible to express $\Lp_V$ in terms of the Dirichlet-to-Neumann map of a one-dimensional Schrödinger operator on the half-line, as we now show. We start by considering the one-dimensional boundary value problem
\begin{equation}   \label{e:schrodinger_prob}
\begin{cases}
   - v_z''  + Q   v_z = -z^2 v_z ,    \qquad \text{ in }   \IR_+,\\
    v_z    \in  L^2(\IR_+).
\end{cases}
\end{equation}
When the potential $Q$ belongs to $L^1(\IR_+)$, it is known that there exists $\beta_Q>0$ such that problem \eqref{e:schrodinger_prob} has a unique solution when $z\in\IC_+\setminus [0,\beta_Q]$.
One then defines the \textit{Weyl-Titchmarsh function} of $Q$ as:
\begin{equation*}
    m_Q(-z^2) := \frac{v'_z(0)}{v_z(0)}.
\end{equation*}

When $Q \in L^1(\IR_+)$, it has been shown in Simon's seminal paper \cite{IST1} that there exists a function $A_Q\in L^1_\loc(\R_+)$, called the $A$-amplitude of $Q$, such that
\begin{equation*}
    m_Q(-z^2) = -z - \int_0^\infty  e^{-2zt} A_Q(t) \, dt \qquad \text{for all } z \text{ such that }  \Re(z)>z_Q,
\end{equation*}
where $z_Q = \frac{\|Q\|_1}{2}>0$. As a matter of fact, it is proved in \cite{MaMe24} that the $A$-amplitude coincides with the notion of the Born approximation for the inverse problem of recovering a potential from its Weyl-Titchmarsh function.

The connection of $m_Q$ with the radial Dirichlet-to-Neumann map $\Lp_V$ defined in \Cref{sec:DtNp} is as follows. Given a potential $V \in L^p_\rad(\IB^d, \R)$, with $p>d/2$ such that zero is not a Dirichlet eigenvalue of $-\Delta+V$, and a spherical harmonic $Y_k \in\gH_{k,d}$, then the unique solution $u\in H^1(\IB^d)$ of
\begin{equation*}
\begin{cases}
    (-\Delta  +V) u &=0,    \quad \text{ in }   \IB^d,\\
    \hfill u|_{\IS^{d-1}}  &=Y_k,  
\end{cases} 
\end{equation*}
can be written in the form
\begin{equation*}
u(x)= |x|^{-\nu_d}v_{k+ \nu_d}(-\log|x|)Y_k(x/|x|), 
\end{equation*}
where, for all $k \in \N_0$, $v_{k+\nu_d} \in L^2(\R_+)$ solves the boundary value problem \eqref{e:schrodinger_prob}  with potential $Q=Q_V$ given by any of the following equivalent identities: if $V=q_V(|\cdot|)$ then
\begin{equation} \label{e:q_to_Q}
    Q_V(t) :=   e^{-2t}q_V(e^{-t}), \; t\in \R_+,\qquad \text{and}  \qquad q_V(r) = r^{-2}Q_V(-\log r), \; r\in (0,1].
\end{equation}
Then, a direct computation shows that $m_{Q_V}$ determines the spectrum of $\Lp_V$ by the relation
\begin{equation} \label{e:lambda_m}
    \lp_k[V] =  -m_{Q_V}\left ( -(k +\nu_d)^2  \right )-\nu_d \qquad \forall k\in \N_0.
\end{equation}

In \cite[Theorem 1]{Radial_Born} it was proved that there exists a $k_V>0$ such that
\begin{equation}\label{e:specbp}
    \sigma_k[\vb] = \lp_k[V] -k  ,\qquad  \forall k \in \N_0,\;k>k_V,
\end{equation}
where $\vb$ is defined in terms of the $A$-amplitude as
\begin{equation} \label{e:A_to_VB}
    \vb (x) := |x|^{-2}A_{Q_V}(-\log |x|), \qquad  x\in \IB^d\setminus\{0\}.
\end{equation}
In addition, it is also proved in \cite[Theorem 1]{Radial_Born} that $|\cdot|^{2k}\vb\in L^1(\IB^d)$. 

The main objective in this section is to refine these results when $V$ arises from a conductivity, in order to prove \Cref{p:Vb_space}. The important formula \eqref{e:A_to_VB} which connects the Born approximation $\vb$ to the $A$-amplitude of the potential $Q_V$ will allow us to extract precise information on the singularities of $\vb$ at $0$ from certain asymptotic properties of $A_{Q_V}$.

\begin{remark} \label{r:characterization}
    The work of Remling \cite{Remling03} implies that not all locally integrable functions can be $A$-amplitudes of Schrödinger operators. 
    This implies, together with identities \eqref{e:from_comp} and \eqref{e:A_to_VB}, that non-trivial necessary conditions are required to characterize the subset of $L^1_\rad(\IB^d;\R)$ formed by the Born approximations $\gb$ corresponding to conductivities in $W^{2,p}_\rad(\IB^d)$. This means that \Cref{mt:conductivity} provides a partial, but not total, characterization of the set of rotation-invariant DtN maps.
\end{remark}

\subsection{Asymptotics of \texorpdfstring{$A$}{A}-amplitudes}  \label{sec:asympt_A_amplitude}

In this section, we review the main results from \cite{IST3} that will be needed in the sequel. 

Let $Q\in L^1(\IR_+,\IR)$ be such that
\begin{equation} \label{e:Q_condition}
    \int_{0}^\infty |Q(t)|(1+t) \, dt <\infty.
\end{equation}
We define the $1$-$d$ Schrödinger operator $H_{Q}$, acting on $L^2(\IR_+)$ by
\begin{equation} \label{e:H_def_q}
    H_{Q} u(t) := \left(-\frac{d^2 \,}{dt^2} + Q(t) \right) u(t),  \qquad t\in\IR_+,
\end{equation}
with its domain (corresponding to Dirichlet boundary condition at $0$) 
$$
D(H_Q) = \{ u \in L^2(\IR_+)\,:\, u \in \cA\cC_\loc(\IR_+),\, H_Q u \in L^2(\IR_+),\, u(0) = 0\},
$$
where $Q$ is determined by \eqref{e:q_to_Q}. Note that the operator $(H_Q, D(H_Q))$ is self-adjoint on $L^2(\IR_+)$.
\par\noindent

We now summarize well-known facts, see \cite{IST3} and \cite{Marchenko}. Under the previous assumption, for every $z\in \ol{\C_+}$ there exists a unique solution $\psi(t,z)$ of \eqref{e:schrodinger_prob} such that
\begin{equation}\label{e:jostsol}
    \psi(t,z) = e^{-tz}(1+o(1)),\qquad t\to +\infty.    
\end{equation}
These solutions are known as Jost solutions (they are usually defined using the variable $z'= iz$), and $F(z): = \psi(0,z)$ is the Jost function associated to $Q$. In addition, both functions $F$ and $\psi(t,\cdot)$ are  analytic in $\C_+$. 
Notice that the zeros of $F$ at $\C_+$ occur at the points $z$ such that $-z^2$ is a Dirichlet eigenvalue of $H_{Q}$. 

We will later show that, under suitable assumptions on $Q$, the operator $H_{Q}$ has no eigenvalues in $[0,\infty)$. In addition, it always has a finite number of negative eigenvalues $-\kappa_j^2$, where $j=1,\dots,J$, with $J \in \N_0$ and $0<\kappa_J<\dots<\kappa_1$ (see \cite{IST3}). Recall that the Jost function associated to $Q$ satisfies $F(\ka_j) =0$ for all $j\in J$. It will also be relevant if the Jost function vanishes at $z=0$. In fact, as in \cite{IST3},
we will say that $H_{Q}$ has a zero resonance iff $F(0)=0$.

Let $\ell \ge 0$. Define the weighted  $L^1_\ell(\R_+)$ space with norm
\begin{equation*} 
    \norm{f}_{L^1_\ell(\R_+)} := \int_0^\infty |f(t)|(1+t)^\ell\, dt.
\end{equation*}
\begin{theorem}[{\cite[Theorem 1.3]{IST3}}] \label{l:ramm_sim_2}
    Let  $\ell\ge 3$, $Q \in L^1_\ell(\R_+)$ and $(\ka_j)_{j=1}^J$ as above. 
    Then, there exist some constants $(a_j)_{j=0}^J$, and a function $g \in L^1_{\ell-1}(\R_+)$
    such that
    \begin{equation} \label{e:a_resonances}
        A_Q(t) = g(t) + a_0 + \sum_{j=1}^J a_je^{2t\ka_j}.
    \end{equation}
    The constant $a_0$ vanishes if $H_Q$ has no zero resonances, and one always has $ a_1,\dots,a_J\le 0$. As a consequence,
    \begin{equation} \label{e:m_A_function}
        m_Q(-z^2) = -z - \int_0^\infty  e^{-2zt} A_Q(t) \, dt,  
    \end{equation}
    for all  $z$ such that  $\Re(z)>\ka_1$ if the point spectrum of $H_Q$ is non-empty, $\Re(z)>0$ if it is empty, and $\Re(z)\ge 0$ if, in addition, there are no zero resonances. 
\end{theorem}
\begin{proof}
     The identity \eqref{e:a_resonances} is a direct consequence of \cite[Theorem 1.3]{IST3}. The constants satisfy $ a_1,\dots,a_J\le 0$ as follows from \cite[Remark 3, p. 321]{IST3}. The fact that \eqref{e:m_A_function} holds for all $\Re(z) >\ka_1$ follows by the analyticity of both sides of the identity in $\{ z\in \IC: \Re(z)> \ka_1 \}$. In the case of empty point spectrum, the same holds in  $\{ z\in \IC: \Re(z)> 0 \}$. Further, if there are no zero resonances, then $A_Q(t)$ belongs to $L^1(\R_+)$ and both sides of \eqref{e:m_A_function} are continuous in $\{ z\in \IC: \Re(z)\ge 0 \}$, so they must be equal.
\end{proof}

\subsection{Singularities of the Born approximation}
We will now describe how the results of the previous section translate into results on the Born approximation for radial Dirichlet-to-Neumann maps $\Lp_V$.  The following Lemma shows, in particular, that as soon as $p>d/2$ and $V \in L^p_\rad(\IB^d, \R)$ we have that  $Q_V$ satisfies \eqref{e:Q_condition}.

\begin{lemma} \label{l:spaces}
    Let $p>d/2$  and $\alpha \ge 0$. 
   Then, there exists a constant $C(\alpha,p,d) >0$ such that, for every $V\in L^p_\rad(\IB^d)$, the potential $Q_V$ given by \eqref{e:q_to_Q} satisfies
    \begin{equation*}
        \norm{Q_V}_{L^1_\alpha(\R_+)} \le C(\alpha,p,d)    \norm{V}_{L^p(\IB^d)}.
    \end{equation*}
    In particular, $C(0,p,d) =  \left(\frac{p-1}{2p-d} \right)^{1-1/p}$.
\end{lemma}
\begin{proof}
    Let $1/p +1/q =1$ and $c_d:=|\IS^{d-1}|^{-1}$. Then
    \begin{align*}
        \norm{Q_V}_{L^1_\alpha(\R_+)} =  \int_{\IR_+} |Q_V(t)| (1 + t)^\alpha \, dt  &=  c_d \int_{\IB^d} |V(x)| (1+ |\log|x||)^\alpha  |x|^{2-d} \, dx 
        \\ &\le  c_d \norm{V}_{L^p(\IB^d)} \left ( \int_{\IB^d} (1+ |\log|x||)^{q\alpha}  |x|^{q(2-d)} \, dx \right)^{1/q},
    \end{align*}
    and the second integral is finite for all $\alpha \ge 0$ and  $1 \le q< d/(d-2)$, which implies $p>d/2$. Computing explicitly the integral when $\alpha =0$ yields the first inequality.
\end{proof}
As a consequence of this result and \Cref{l:ramm_sim_2} we obtain the following description of the singularity at the origin of the Born approximation.
Recall from previous sections  that
\begin{equation} \label{e:d_alpha_norm}
    \norm{f}_{d,\alpha}:=\int_{\IB^d} |f(x)| |x|^{2-d}  (1-\log|x|)^\alpha \, dx <\infty, \qquad  \alpha \ge 0,\,  d\ge 2.
\end{equation}
\begin{theorem}  \label{t:negative_eigenvalues} 
    Let $d\ge 2$,  $V\in L^p_\rad(\IB^d,\IR)$ with $p>d/2$, and $(\ka_j)_{j=1}^J$ such that $(-\ka_j^2)_{j=1}^J$ is the point spectrum of $H_{Q_V}$, as above. There
    exist  some constants $a_0\in \R, \, a_1,\dots,a_J\le 0$, such that  
    \begin{equation*}
    \vb(x) = \vb_\mathrm{reg}(x) + \frac{a_0}{|x|^2} + \sum_{j=1}^J \frac{a_j }{|x|^{2\kappa_j+2}},
    \end{equation*}
    where
    $\vb_\mathrm{reg}$ is a radial function  satisfying $\norm{\vb_\mathrm{reg}}_{d,\alpha}<\infty $ for all $\alpha>0$.
    \\
    In particular,    $a_0=0$ if $H_{Q_V}$ has no zero resonance.
\end{theorem}
\begin{proof}
    Let $Q_V$ be given by \eqref{e:q_to_Q}. Then $Q \in L^1_\ell(\R_+)$ for all $\ell \ge 1$, by \Cref{l:spaces}. Therefore it satisfies the conditions to apply \Cref{l:ramm_sim_2}. Thus
    \begin{equation*}
        A_{Q_V}(t) = g(t) + a_0 + \sum_{j=1}^J a_je^{2t\ka_j},
    \end{equation*}
    for some constants $(a_j)_{j=0}^J$. Here $g\in L^1_{\ell}(\IR_+)$ for all $\ell \ge 0$. The constant $a_0=0$ if $H_{Q_V}$ has no zero resonances. Thus, the statement for $\vb$ follows directly using \eqref{e:A_to_VB}, with $\vb_\mathrm{reg}(x) : = |x|^{-2}g(-\log |x|)$.
\end{proof}


\Cref{t:negative_eigenvalues} has interesting consequences. In the first place, it shows that $\vb$ can be a very singular object: one can always find smooth radial potentials $V$ such that $H_{Q_V}$ has an eigenvalue $\kappa_1$ as large as desired (for example, a smooth potential with large negative values).  
On the other hand, if one imposes conditions on the potentials that bound the value of $\ka_1$, then the Born approximation can be an integrable function. We will examine this in the following section.
\begin{remark}
If $V\in L^p_\rad(\IB^d,\IR)$, $p>d/2$ and $V(x) \ge 0$ $a.e.$ on $\IB^d$, then $H_{Q_V}$ has no bound states nor zero resonances.
 As a consequence of \Cref{t:negative_eigenvalues} we see that $\norm{\vb}_{d,\alpha}<\infty$ for all $\alpha \ge 0$, which is stronger that $\vb\in L^1(\IB^d)$.
\end{remark}
\begin{remark}
    \Cref{t:negative_eigenvalues} can be improved to take into account the complex resonances of $H_{Q_V}$, which will add oscillatory terms to the asymptotic expansion. If one assumes that   for some $-\infty<\delta<0$ the potential satisfies
      $ \int_{\IB^d} |V(x)| |x|^{2-d+\delta} \, dx <\infty,$ 
   (which always holds if $V\in L^p(\IB^d)$ with $p>d/2$)  then, for a fixed $\delta<\delta'<\infty$,
    \begin{equation*}
        \vb(x) = \sum_{j=1}^J \frac{a_j }{|x|^{2\kappa_j+2}} + \sum_{j=1}^M \frac{b_j}{|x|^{2\lambda_j+2} }+ \sum_{j=1}^N c_j \frac{\cos(\theta_j-2\nu_j \log|x| )}{|x|^{2\mu_j +2}} + \vb_\mathrm{reg}(x),
    \end{equation*}
    where $0<\kappa_J<\dots<\kappa_1$ corresponds to the negative eigenvalues $-\kappa_j^2$ of $H_{Q_V}$, $\delta' <\lambda_M \le \dots \le \lambda_1 \le 0$ corresponds to the real resonances larger than $\delta'$, $\mu_j \pm i \nu_j$ with $\delta' <\mu_j<0$ correspond for all $j = 1,\dots, N$ to the complex resonances  with real part larger than $\delta'$, and, finally, $\vb_\mathrm{reg}(x) $ is a radial function  satisfying that
    $\int_{\IB^d} |x|^{2-d + \delta'}|\vb_\mathrm{reg}(x)| \, dx<\infty$. This follows from \cite[Theorem 1.4]{IST3} using the same strategy as in the proof of  \Cref{t:negative_eigenvalues}. This oscillatory behavior has been observed in  the numerical reconstructions of $\vb(x)$ in \cite{BCMM23_n}.
\end{remark}
Another consequence of \Cref{l:ramm_sim_2} is the following.
\begin{lemma} \label{l:N_set}
     Let $d\ge 2$,  $V\in L^p_\rad(\IB^d,\IR)$ with $p>d/2$. Then for every $k\in \cN$
    \begin{equation*} 
        \int_{\IB^d} |\vb(x)| |x|^{2k} \, dx <\infty, \quad \text{and}\quad \sigma_k[\vb] = \lp_k[V]-   k,
    \end{equation*}
    where $\cN \subseteq \N_0$ is the set defined as follows:
        \begin{enumerate}[i)]
        \item  If the point spectrum of $H_{Q_V}$ is non-empty, 
        $\cN = \{k\in \N_0: k> \ka_1 -\nu_d \}$.
        \item If  the point spectrum of $H_{Q_V}$ is empty then:
        \begin{enumerate}
            \item If $d>2$,  $\cN = \N_0$.
            \item If $d=2$, $\cN = \N_0$ provided that $H_{Q_V}$ has no zero resonance, and $\cN = \N$ otherwise.
        \end{enumerate}
    \end{enumerate}
\end{lemma}
\begin{proof}
    Let $Q_V$ be given by \eqref{e:q_to_Q}. Then $ Q_V\in L^1_\ell(\R)$ for all $\ell \ge 3$ by \Cref{l:spaces}, so \cref{l:ramm_sim_2} holds.
    Then, by \eqref{e:lambda_m}, \eqref{e:m_A_function} and \eqref{e:A_to_VB} in that order:
    \begin{multline*}
        \lp_k[V] -k  
        = -m_{Q_V}\left ( -(k +\nu_d)^2  \right )-(k+\nu_d)  \\
        =  \int_0^\infty  e^{-2kt}  A_{Q_V}(t) e^{-(d-2)t} \, dt 
        = \frac{1}{|\IS^{d-1}|}\int_{\IB^d} \vb(x) |x|^{2k} \, dx =\sigma_k[\vb] ,
    \end{multline*}
    where the integrals are always absolutely convergent.
    Notice that, by \Cref{l:ramm_sim_2}, the second equality holds when $k +\nu_d>\ka_1$ if the spectrum of $H_{Q_V}$ is non-empty, when $k +\nu_d>0$ if it is empty, and when $k +\nu_d \ge 0$ if, in addition, there are no zero resonances. This yields all the cases enumerated in the statement.
\end{proof}

\subsection{The conductivity case \texorpdfstring{and proof of \Cref{p:Vb_space}}{}}    \label{sec:specconduc}

If the potential $V$ arises from a conductivity, then the operator $H_{Q_V}$ is always uniformly bounded  from below and $\kappa_1< \nu_d$, as we now show. As a consequence, the singularities of the Born approximation will be of  integrable type in this case. We first need the following lemma.
\begin{lemma} \label{l:0eigen}
Assume $Q_V\in {L}^1_2(\IR_+)$. 
Then $0$ is not an eigenvalue of $H_{Q_V}$.
\end{lemma}
\begin{proof}
At zero energy, any solution $u$ to $H_{Q_V}u=0$ on $\IR_+$ satisfies the Volterra equation
\[
u(x)=A+Bx+\int_x^\infty (t-x)\,Q_V(t)\,u(t)\,dt.
\]
Writing $w(x)=u(x)-(A+Bx)$, we obtain, using that $t-x\le t$ and $(1+t^2)Q_V\in L^1(\IR_+)$,
\[
|w(x)|\le (|A|+|B|)\int_x^\infty t(1+t)|Q_V(t)|\,dt
 +\!\int_x^\infty t|Q_V(t)|\,|w(t)|\,dt.
\]
Gronwall’s lemma gives:
\begin{equation*}
    |w(x)|\leq (|A|+|B|)M(x)\exp\left(\int_x^\infty t|Q_V(t)|\,dt\right), \qquad M(x):=\int_x^\infty t(1+t)|Q_V(t)|\,dt,
\end{equation*}
since $M(x)$ is a  decreasing function. This shows that $w(x)\to0$ as $x\to\infty$. If $u\in L^2(\IR_+)$ this forces $A=B=0$, and therefore $w=0$.
\end{proof}

\begin{lemma}  \label{l:eigenvalues_cond}
    Let $p>d/2$ and $\gamma \in W^{2,p}_\rad(\IB^d)$ be a  conductivity, and let $V\in L^p_\rad(\IB^d)$ be given by \eqref{e:V_to_gamma}. Let $H_{Q_V}$ be the associated Hamiltonian defined in \eqref{e:H_def_q}. Then
    \begin{enumerate}[i)]
        \item   If $d>2$, the point spectrum of $H_{Q_V}$ is contained in $(-\nu_d^2,0)$.
        \item   If $d=2$, the point spectrum of $H_{Q_V}$ is empty and it has no zero resonance.
    \end{enumerate}
\end{lemma}
\begin{proof}
    Let $\gamma(x)$ such that $\ga(x) = \ga_0(|x|)$, and define $\phi(t) := e^{-\nu_d t}\sqrt{\ga_0(e^{-t})}$.
    From \eqref{e:V_to_gamma}, \eqref{e:H_def_q},  \eqref{e:q_to_Q} and straightforward computations, it follows that 
    \begin{equation*}
        Q_V(t) =\frac{\phi''(t)}{\phi(t)}-\nu_d^2.
    \end{equation*}
    Let $w(t):=\frac{\phi'(t)}{\phi(t)}=(\log \phi(t))'$. 
    Then $Q_{V}$ can be written explicitly in terms of $w$ as
    \[
    Q_V(t) =
        \left(\frac{\phi'(t)}{\phi(t)}\right)'+ \left(\frac{\phi'(t)}{\phi(t)}\right)^2-\nu_d^2
         = w'(t)+w(t)^2-\nu_d^2.
    \]
    With this representation for $Q_V$, one can prove that one has the factorization
    \[
    H_{Q_{V}}+\nu_d^2 = B^{*}B,\qquad 
    B:=-\partial_t +w(t), \quad B^{*}=\partial_t+w(t).
    \]
    Consequently, for every $u\in H^1_0(\R_+)$, one has that
    \begin{multline*}
      (u,(H_{Q_V}+\nu_d^2)u)_{L^2(\R_+)} =  \int_{0}^{\infty}
      |-\partial_t u(t)+w(t)u(t)|^{2}
      \,dt = \int_{0}^{\infty}\phi(t)^{2}\left|\left(\frac{u(t)}{\phi(t)}\right)'\right|^{2}\,dt \geq 0.
    \end{multline*}
    Notice that $(u,(H_{Q_V}+\nu_d^2)u)_{L^2(\R_+)}$ is always positive in $H^1_0(\R_+)$ since the second integral only vanishes for $u =c\phi$ with $c\in\IC$, but $\phi(0)=\sqrt{\ga_0(1)}>0$, hence $\phi \notin H^1_0(\R_+)$. Therefore
    \[
    \Spec_{H^1_0}(H_{Q_{V}}) \subseteq (-\nu_d^2,\infty).
    \]
    By \Cref{l:0eigen}, $0$ is not an eigenvalue $H_{Q_V}$. By \cite[Lemma 3.1.1]{Marchenko}, for $\ka \in \R$ the solutions of $ - v''  + Q_V   v = \ka^2 v_z$ has two independent solutions $v_\pm$ that behave as $e^{\pm i \ka  t}$ for $t\to +\infty$, so $\ka^2 \in (0,\infty)$ can never  be an eigenvalue of $H_{Q_V}$. 
    This ensures the absence of embedded eigenvalues in $[0,\infty)$ (since $Q_V \in L^1_\alpha(\IR_+) \subset L^1(\IR_+)$, one can also apply the classical result of Kneser--Weyl~\cite{Kneser1926,Weyl1910} --see also Simon~\cite{Simon18}-- to prove the absence of embedded eigenvalues in $(0,\infty)$). 
    Hence,
    it follows that the point spectrum of $H_{Q_V}$ must be contained in $(-\nu_d^2,0)$ if $d>2$, and must be empty if $d=2$.

    It remains to prove that there is not a zero resonance when $d=2$. Notice that
    \[
      (H_{Q_{V}}+\nu_d^{2})\phi=0 ,\qquad d\ge 2,
     \] 
     so, if $d=2$, we have $H_{Q_{V}}\phi=0$. Since $\phi(t) \to \sqrt{\ga(0)}>0$ as $t \to \infty$, then
    \[
         \psi(t,0) = \frac{1}{\sqrt{\ga(0)} } \phi(t) ,
    \]
    is the $z=0$ Jost solution associated to $H_{Q_{V}}$ that was introduced in \eqref{e:jostsol}. 
     The Jost solution is unique and we have $\sqrt{\ga(0)}\psi(0,0) = \phi(0) = \sqrt{\ga_0(1)}>0$. 
     This implies that $\psi(0,0)>0$ and that $H_{Q_{V}}$ has no Dirichlet zero resonance.
\end{proof}

We can now prove \Cref{p:Vb_space}.
\begin{proof}[Proof of \Cref{p:Vb_space}] 
    If $V$ is given by \eqref{e:V_to_gamma}, then $V \in L^{p}_\rad(\IB^d,\R)$, since $\gamma(x) >c>0$ in $\IB^d$. 
    From \Cref{t:negative_eigenvalues}, \Cref{l:eigenvalues_cond}, and \Cref{l:N_set} it follows there exists a solution $\vb$ of the moment problem \eqref{e:moment_prob_pot}. 
    This solution is unique as follows from  \cite[Equation 1.20]{BCMM22}, which is reproduced here in \eqref{e:ft_moments}. 
    Assertions i) and ii) are immediate from \Cref{t:negative_eigenvalues}, \Cref{l:N_set}, and \Cref{l:eigenvalues_cond}. 
    Assertion iii) is proved in \cite[Theorem~5]{Radial_Born}.
    Finally, iv) is an immediate consequence of i) and ii).
\end{proof}
\subsection{Explicit examples}\label{sec:ex}

Here we present a family of conductivities for which the Born approximation can be computed explicitly. 

Let $d\geq 2$, $\mu>0$ and $\nu\geq 0$. Define
\begin{equation*}
    \rho_{d,\mu,\nu}(r):=\begin{cases}\displaystyle\frac{1}{\nu+\nu_d}\left(\frac{2\mu}{1+\alpha r^{2\mu}} +\nu_d-\mu\right), & (d,\nu)\neq(2,0)\medskip\\
    \displaystyle\mu\frac{1-r^{2\mu}}{1+ r^{2\mu}}, & (d,\nu)=(2,0)
    \end{cases}
    ,\qquad 
    \alpha := \frac{\mu-\nu}{\mu+\nu},
\end{equation*}
and note that $\rho_{d,\mu,\nu}(r)>0$ for every $r\in (0,1)$.
\begin{proposition}
    Let $d\geq 2$, $\mu>0$ and $\nu\geq 0$ ($\nu>0$ when $d=2$) and set 
    \begin{equation}
        \ga_{d,\mu,\nu}(x):=\rho_{d,\mu,\nu}(|x|)^2.
    \end{equation}
    Then $\ga_{d,\mu,\nu}$ is a conductivity in $W^{2,p}(\IB^d)$ for some $p>d/2$ (every $p$ when $\mu\geq1$) such that $\ga_{d,\mu,\nu}|_{\IS^{d-1}}=1$ and whose Born approximation is given by
    \begin{equation}\label{e:gbex}
        \gb_{d,\mu,\nu}(x)=
        \begin{cases}
        \displaystyle
        1+\frac{(\nu^{2}-\mu^{2})}{\nu(\nu+\nu_d)}(|x|^{2\nu}-1), & \nu>0,\\[10pt]
        \displaystyle
        1-\frac{2\mu^{2}}{\nu_d}\,\log|x|, & \nu=0,\; d\geq 3.
        \end{cases}
    \end{equation}
\end{proposition}
\begin{proof}
    Direct computation shows that:
    \begin{equation*}
        \Delta\sqrt{\ga_{d,\mu,\nu}}=V\sqrt{\ga_{d,\mu,\nu}}, \qquad V(x):=-8\mu^2 \alpha \frac{|x|^{2(\mu-1)}}{\left(1+ \alpha |x|^{2\mu} \right)^2}. 
    \end{equation*}
    In \cite{Radial_Born}, it is shown that the Born approximation of $V$ is given by 
    \begin{equation*}
    \vb (x) =  2(\nu^2-\mu^2)|x|^{2(\nu-1)}.
    \end{equation*}
    These potentials correspond, after performing the change of variables \eqref{e:q_to_Q}, to those studied in \cite[Section~11]{IST2}. One can also check, by direct computation, that if $\gb$ is given by \eqref{e:gbex} then
    \begin{equation*}
        \Delta \gb_{d,\mu,\nu} = 2\vb,\qquad \gb_{d,\mu,\nu}|_{\IS^{d-1}}=1.
    \end{equation*}
    The claim then follows from relation \eqref{e:from_comp} and uniqueness of solutions to the Poisson equation.
\end{proof}
This example illustrates some of the results in this section.
\begin{remark} \
    \begin{enumerate}[i)] 
        \item The parameter $-\nu$ represents the unique real resonance of $H_{Q_V}$, (see  \cite{IST2}, case 2, p. 636).
        \item The presence of the zero resonance for $d\geq 3$ results in a Born approximation $\gb$ with a logarithmic singularity at the origin (and $\vb$ with a $|x|^{-2}$ singularity), as predicted by \Cref{mt:gb_structure} (see \cref{f:ex_2}).
        \item In the case $d=2$ and $\nu =0$ there is a zero resonance and  $\ga_{2,\mu,0}$ is a degenerate conductivity, since it vanishes on $\IS^1$. This is consistent with \Cref{l:eigenvalues_cond}~ii).
    \end{enumerate}
\end{remark}

\subsection{\texorpdfstring{$\vb$}{VB} and spectral asymptotics of \texorpdfstring{$\Lp_V$ }{the DtN map}}\label{sec:spa} 

To conclude this section, we prove a result analogous to \Cref{c:specrig} in the context of DtN associated to Schrödinger operators.
\begin{theorem}\label{thm:spa}
Let $d\geq 2$, $p>d/2$ and choose $s\in(0,1)$. For any two potentials $V_1,V_2\in L^{p}_\rad(\IB^d,\IR)$ such that $0\notin \Spec_{H^1_0(\IB^d)} (-\Delta+V_j)$, $j=1,2$, the following are equivalent.
\begin{enumerate}[i)]
    \item $V_1|_{U_s}=V_2|_{U_s}$.
    \item The following spectral asymptotics hold for $\Lp_{V_1}-\Lp_{V_2}$:
    \begin{equation*}
        \lp_k [V_1]-\lp_k [V_2]=\cO(s^{2k}),\qquad k\to\infty.
    \end{equation*}
\end{enumerate}
\end{theorem}
 The proof is based on \Cref{l:mom} and results from \cite{Radial_Born}. In fact, following ideas from that paper, one could remove the assumption  $0\notin \Spec_{H^1_0(\IB^d)} (-\Delta+V_j)$ and assume weaker regularity conditions than $V_j\in L^{p}_\rad(\IB^d,\IR)$.
\begin{proof}
	Start by noting that \eqref{e:specbp} implies that
	\begin{equation}\label{e:diffpot}
	\lp_k [V_1]-\lp_k [V_2] = \sigma_k[\vb_1-\vb_2],\qquad k>k_0:=\max(k_{V_1},k_{V_2}).
	\end{equation}
    Then, if 
    \begin{equation*} \label{e:diffpot_2}
        F(x) : = (\vb_1(x)-\vb_2(x))|x|^{2k_0},
    \end{equation*}
    one has that $F\in L^1_\rad(\IB^d)$.
    If assumption i) holds, then by  \cite[Theorem 2]{Radial_Born} $\vb_1|_{U_s}=\vb_2|_{U_s}$, which yields $F|_{U_s} =0$.
    Then, 
	\Cref{l:mom} implies that $C>0$ exists such that, for every $k \ge k_0$
	\[
	|\sigma_{k-k_0}[F]|\leq Cs^{2(k-k_0)},
	\]
	 which suffices to conclude ii) after noticing that $\sigma_{k-k_0}[F] = \sigma_k[\vb_1-\vb_2]$.
	
	To prove the converse, note that ii) implies, in view of \eqref{e:diffpot} and \eqref{e:diffpot_2}, that $\sigma_k[F]=\cO(s^{2k})$ as $k\to\infty$. Then,
	\Cref{l:mom} implies $F|_{U_s} =0$. As a consequence, $\vb_1|_{U_s}=\vb_2|_{U_s}$ a.e. and \cite[Theorem~2]{Radial_Born} implies that $V_1|_{U_s}=V_2|_{U_s}$ as claimed.
\end{proof}

\section{Stability properties of the Born approximation}\label{sec:stability}

\subsection{Hölder stability}

In the Calderón problem one can prove  estimates of the kind
\begin{equation} \label{e:stab_calderon}
    \norm{\ga_1-\ga_2}_{L^q(\IB^d)} \le  \om_\cK \left( \norm{\La_{\ga_1}-\La_{\ga_2}}_{\cL(H^{1/2},H^{-1/2})} \right), \qquad \ga_1,\ga_2 \in \cK,
\end{equation}
where $\cK$ is a compact subset of $L^q(\IB^d)$, $1\le q\le \infty$, and $\om_\cK$ is a modulus of continuity \cite{Alessandrini88,Barcelo_Barcelo_Ruiz_01,Clop_Faraco_Ruiz_10, MR3062871}. Assuming the set $\cK$ only contains radial functions, this becomes
\begin{equation*}
    \norm{\ga_1-\ga_2}_{L^q(\IB^d)} \le  \om_\cK \left( \sup_{k\in \N}|\la_k[\ga_1]-\la_k[\ga_2]|  \right),  \qquad \ga_1,\ga_2 \in \cK.
\end{equation*}

Due to the strong continuity properties of the moment problem,  the previous estimate implies
\begin{equation} \label{e:log_mod}
    \norm{\ga_1-\ga_2}_{L^q(\IB^d)} \le  \om_\cK \left(\norm{\gb_1 -\gb_2}_{W^{2,1}(\IB^d)} \right),
\end{equation}
provided that $\ga_1,\ga_2 \in \cK \cap W^{2,p}(\IB^d)$, for some $p>d/2$, and that $\ga_1|_{\IS^{d-1}} = \ga_2|_{\IS^{d-1}}$. Indeed, by \Cref{mt:conductivity} and \eqref{e:moments_0} we have
\begin{align*}
    |\la_k[\ga_1]-\la_k&[\ga_2]|  \le  \frac{1}{2}\left|\sigma_k[ \Delta(\gb_1-\gb_2)]  \right| + \frac{1}{2}\left|\p_\nu \ga_1|_{\IS^{d-1}} -\p_\nu \ga_2|_{\IS^{d-1}} \right| \\
     &= \frac{1}{2}\left|\sigma_k[ \Delta(\gb_1-\gb_2)]  \right| + \frac{1}{2}\left|\sigma_0[ \Delta(\gb_1-\gb_2)] \right| \\
     &\le \sup_{k\in \N_0} \left|\sigma_k[ \Delta(\gb_1-\gb_2)] \right|  \le \norm{\Delta(\gb_1-\gb_2)}_{L^1(\IB^d)} \le \norm{\gb_1-\gb_2}_{W^{2,1}(\IB^d)}.
\end{align*}

It is well known that in general $\om_\cK$ in \eqref{e:stab_calderon} cannot be taken better than logarithmic when $\cK$ includes a ball in a Sobolev space, \cite{Mand00,KRS21}. \Cref{mt:stability-gamma} improves estimate \eqref{e:log_mod} to a Hölder modulus of continuity, and gets rid of the assumption $\ga_1|_{\IS^{d-1}} = \ga_2|_{\IS^{d-1}}$.

\subsection{Proof of \texorpdfstring{\Cref{mt:stability-gamma}}{Theorem 3}}

To get the estimate for the conductivity case, we use the following lemma, and two stability results for the Schrödinger problem from \cite{Radial_Born}, which in turn are based on stability properties of the $A$-amplitude \cite[Theorem 5.1]{Radial_Born}.
Let $0<s<1$ and recall the definition of $U_s$ in \eqref{e:us_def}.
\begin{lemma} \label{l:holder_conductivity}
    Let $K>1$, $N>0$, $d/2<p <\infty $, $0<s<1$. There exist some constants $C_{K,N,p,d,s},C_{K,N,p,d}>0$ such that,  for all $\ga_j \in W^{2,p}_\rad(\IB^d)$,  satisfying
    \begin{equation} \label{e:cond_lemma_ga}
        K^{-1} \le \ga_j(x) \le K, \qquad  \norm{\ga_j}_{W^{2,p}(\IB^d)} \le N, \quad j=1,2,
    \end{equation}
    then, if $a_j: = \ga_j|_{\IS^{d-1}}$, $b_j:= \frac{1}{2}\partial_\nu  \ga_j|_{\IS^{d-1}}$ one has the local estimate
    \begin{equation*}
       C_{K,N,p,d,s}^{-1}\norm{\ga_1-\ga_2}_{W^{2,1}(U_s)}  \le \norm{V_1-V_2}_{L^1(U_s)} + |a_1-a_2| +|b_1-b_2|.
    \end{equation*}
    and the global estimate
    \begin{equation*}
        C_{K,N,p,d}^{-1}\left(\norm{\ga_1-\ga_2}_{W^{2,1}(\IB^d)}  + \norm{\ga_1-\ga_2}_{H^1(\IB^d)} \right) \le  \norm{V_1 -V_2}_{L^1(\IB^d)} ^{\beta} + |a_1-a_2|,
    \end{equation*}
    for $\beta = \frac{p-r_d}{r_d(p-1)}$, with $r_d:=\frac{2d}{d+2}$ for $d \ge 3$, $r_2:= (p+1)/2$.
\end{lemma}
The proof of the lemma is rather technical but standard.
For the sake of completeness, we include it in \Cref{sec:appendix}.

\begin{theorem}[Theorem 3 of \cite{Radial_Born}] \label{t:stability_local_V}
    Let $d\ge 2$, $d/2<p\le \infty $, and $0<s<1$. For all $N>|\IS^{d-1}|^{1/p}$, there exist  constants $C_{d,N,p,s}>0$ and $0<\delta_{d,p,s}<1$ such that, for all  $V_1,V_2 \in L^{p}_\rad (\IB^d;\R)$ satisfying
    \begin{equation} \label{e:pot_bound_local}
        \max_{j=1,2} \norm{V_j}_{L^p(U_s)} \le N, \qquad \qquad     \norm{\vb_1-\vb_2}_{L^1(U_s)} <\delta_{d,p,s} ,
    \end{equation}
    one has
     \begin{equation}  \label{e:loc_sta}
       \norm{V_1-V_2}_{L^1(U_s)} < C_{d,N,p,s} \, \norm{\vb_1-\vb_2}_{L^1(U_s)}^{\frac{p-1}{2p-1}}.
    \end{equation}
\end{theorem}
\begin{theorem}[Theorem 4 of \cite{Radial_Born}] \label{t:stability_global_V}
    Let $d\ge 2$, $d/2<p< \infty $ and $M > d-1$. 
    There exists $C_{d,M,p}>0$ and  $0<\alpha<1$, depending only on $M,d,p$, such that for every $V_1,V_2 \in L^{p}_\rad (\IB^d;\R)$ satisfying
    \begin{equation} \label{e:def_pot_bound}
        \begin{aligned}
            &\max_{j=1,2} \norm{V_j}_{L^p(\IB^d)} < \frac{|\IS^{d-1}|^{\frac{1}{p}}}{4} \frac{p-d/2}{p-1} M ,  \\
            &\int_{\IB^d}|\vb_1(x)-\vb_2(x)| |x|^{M-(d-2)} \,dx< |\IS^{d-1}|,
        \end{aligned}
    \end{equation}
    one has
    \begin{equation} \label{e:glob_sta}
       \norm{V_1-V_2}_{L^1(\IB^d)} <C_{d,M,p}\left(\int_{\IB^d}|\vb_1(x)-\vb_2(x)||x|^{M-(d-2)}\, dx\right)^\alpha.
    \end{equation} 
\end{theorem}
\begin{lemma}[Lemma 5.6 of \cite{Radial_Born}] \label{l:vb_bound}
    Let $d\ge 2$, $M>d-2$, and  $d/2<p<\infty$. There exists a constant $\beta(d,p)>0$ such that if  $M\ge \beta(d,p) \norm{V}_{L^p(\IB^d)} $, then 
    \begin{equation*}
        \int_{\IB^d}|\vb(x)| |x|^{M-(d-2)} \,dx \le 2|\IS^{d-1}|\beta(d,p) \norm{V}_{L^p(\IB^d)}.
    \end{equation*}
\end{lemma}

\begin{proof}[Proof of \Cref{mt:stability-gamma}]   
   Since the conductivities are radial and \eqref{e:traces} holds, by \Cref{l:w_rad}~i) we have for some $C>0$ that
    \begin{equation} \label{e:boundary_W21}
        |a_1 - a_2| + |b_1-b_2| \le C\norm{\gb_1-\gb_2}_{W^{2,1}(\IB^d)},
    \end{equation}
     where  $a_j: = \ga_j|_{\IS^{d-1}}$, $b_j:= \frac{1}{2} \partial_\nu  \ga_j|_{\IS^{d-1}}$. Locally, the same bound holds:
    \begin{equation} \label{e:boundary_W21_2}
        |a_1 - a_2| + |b_1-b_2| \le \frac{C(d,s)}{1-s}\norm{\gb_1-\gb_2}_{W^{2,1}(U_s)}.
    \end{equation}
     Using the previous estimates, the theorem follows combining \Cref{l:holder_conductivity} with Theorems \ref{t:stability_local_V} and \ref{t:stability_global_V} as we now show.
      
     Conditions \eqref{e:ga_conditions} on $\ga_j$, $j=1,2$ imply a corresponding bound for $\norm{V_j}_{L^{p}}$ using  that $\Delta \sqrt{\ga_j} = V_j \sqrt{\ga_j}$, so the first estimates in \eqref{e:pot_bound_local} and \eqref{e:def_pot_bound} hold.
     We now consider the global and local cases separately.
     
    \textit{Proof of local estimate i)}.  
    First we claim that for any given $0<\varepsilon<1$, there exists a $0<\delta<1$ (depending on $\varepsilon,d,K,N,s$) such that $\norm{\vb_1-\vb_2}_{L^1(U_s)} <\varepsilon $ holds  if $\norm{\gb_1-\gb_2}_{W^{2,1}(\IB^d)} < \delta$ holds. 
    This can be proved using \eqref{e:from_comp} together with $ K^{-1} \le \ga_j(x) \le K$ which imply that
     \begin{multline*}
          \norm{\vb_1-\vb_2}_{L^1(U_s)} = \norm{(2a_1)^{-1}(2a_1\vb_1-2a_2\vb_2) + a_2\vb_2(a_1^{-1}-a_2^{-1})}_{L^1(U_s)} \\ \le  \frac{K}{2} \norm{\Delta \gb_1- \Delta\gb_2 }_{L^1(U_s)} + |a_1^{-1}-a_2^{-1}| |a_2|\norm{\vb_2 }_{L^1(U_s)}, 
     \end{multline*}
     we now use that $|a_1^{-1}-a_2^{-1}| \le C_K|a_1-a_2|$, since $K^{-1} \le a_j \le K$, \Cref{l:vb_bound}  to bound the norm of $\vb_2$ (the weight in the lemma is harmless since we are in $U_s$), and \eqref{e:boundary_W21_2}, to obtain
     \begin{equation} \label{e:differences_born}
          \norm{\vb_1-\vb_2}_{L^1(U_s)}\le C(d,K,N,s)  \norm{ \gb_1- \gb_2 }_{W^{2,1}(U_s)}, 
     \end{equation}
     which proves the claim.  This proves that \eqref{e:pot_bound_local} holds, so that the conditions to apply \Cref{t:stability_local_V} are met. 
     Inserting \eqref{e:differences_born} and the local estimate in \Cref{l:holder_conductivity} in \eqref{e:loc_sta} yields
    \begin{equation*} 
       \norm{\ga_1-\ga_2}_{W^{2,1}(U_s)} < C_{N,K,d,p,s} \, \norm{\gb_1-\gb_2}_{W^{2,1}(U_s)}^{\frac{p-1}{2p-1}} +  |a_1 - a_2| + |b_1-b_2|.
    \end{equation*}
    Using \eqref{e:boundary_W21_2} finishes the proof of the estimate.
    
    \textit{Proof of global estimate ii)}. Let $M>d-1$.
    As in the local case we have
    \begin{multline*}
         \int_{\IB^d}|\vb_1(x)-\vb_2(x)| |x|^{M-(d-2)} \,dx   \le  \frac{K}{2} \int_{\IB^d}|\Delta \gb_1(x)- \Delta\gb_2(x)| |x|^{M-(d-2)} \,dx  \\  + |a_1^{-1}-a_2^{-1}| |a_2| \int_{\IB^d}|\vb_2(x)| |x|^{M-(d-2)} \,dx , 
     \end{multline*}
      By \Cref{p:Vb_space}~iv) and \eqref{e:from_comp}, $\Delta \gb_j \in L^1(\IB^d)$ for $j=1,2$. Hence $|x|^{M-(d-2)}$ can be removed in the first integral in the RHS since $M>d-2$. Using \Cref{l:vb_bound},  \eqref{e:ga_conditions}, \eqref{e:boundary_W21}  and the same arguments used to prove \eqref{e:differences_born} we obtain
         \begin{equation} \label{e:differences_born_2}
          \int_{\IB^d}|\vb_1(x)-\vb_2(x)| |x|^{M-(d-2)} \,dx 
          \le C(d,K,N,M)  \norm{ \gb_1- \gb_2 }_{W^{2,1}(\IB^d)},
     \end{equation}
     for $M$ large enough. We can now fix the value of $M$ (the size of $M$ depends on the bounds on $\norm{V_j}_{L^p(\IB^d)}$, and hence on $N$ that controls the size of conductivities).
     Therefore, there exists a $\delta$ such that, if $\norm{\gb_1-\gb_2}_{W^{2,1}(\IB^d)} < \delta$ holds, then  \eqref{e:def_pot_bound} holds. 
     This proves that the conditions to apply \Cref{t:stability_global_V} are met.
     Inserting \eqref{e:differences_born_2} and the global estimate of \Cref{l:holder_conductivity} in \eqref{e:glob_sta} yields
    \begin{equation*}
       \norm{\ga_1-\ga_2}_{H^1(\IB^d)} + \norm{\ga_1-\ga_2}_{W^{2,1}(\IB^d)} < C_{N,K,d,p} \, \norm{\gb_1-\gb_2}_{W^{2,1}(\IB^d)}^{\alpha \beta} +  |a_1 - a_2|.
    \end{equation*}
    The Hölder exponent of this estimate is just the product of the one of \Cref{t:stability_global_V}, with the one of \Cref{l:holder_conductivity}. 
    Finally, one can use \eqref{e:boundary_W21} to bound $|a_1-a_2|$, and that, for all $1<q<p$,
    \begin{equation*}
        \norm{\ga_1-\ga_2}_{W^{2,q}(\IB^d)}\le C_{N,d,p,q}\norm{\ga_1-\ga_2}_{W^{2,1}(\IB^d)}^{\frac{p-q}{q(p-1)}},
    \end{equation*}
    by interpolation, since $\norm{\ga_j}_{W^{2,p}(\IB^d)} \le N$, with $j=1,2$. This finishes the proof of the Theorem. 
\end{proof}

 \appendix

\section{Proof of \texorpdfstring{\Cref{l:holder_conductivity}}{Lemma 4.1}} \label{sec:appendix}

    \textit{Global estimate:} 
    In this proof, $C$ will always denote a constant that depends on $d,p,K,N$.
    
    Let $v= \sqrt{\ga_1}-\sqrt{\ga_2}$. Then, by \eqref{e:V_to_gamma} one has that $-\Delta v +V_1v = -\sqrt{\ga_2} (V_1-V_2) $. Thus, if $u: = v\ga_1^{-1/2}$, then $\nabla \cdot (\ga_1 \nabla u) = \sqrt{\ga_1\ga_2} (V_1-V_2),$ with $u|_{\IS^{d-1}} = 1-\sqrt{\ga_2/\ga_1}$. By the uniform bounds in the ellipticity of $\ga_j$, we have
    \begin{equation*}
        C^{-1}\norm{1-\sqrt{\ga_2/\ga_1}}_{H^1(\IB^d)} \le  \norm{V_1-V_2}_{H^{-1}(\IB^d)} + \norm{1-\sqrt{a_2/a_1}}_{H^{1/2}(\p \IB^d)}.
    \end{equation*}
    Using that the trace at the boundary is constant, and defining $r_d:=\frac{2d}{d+2}$ for $d \ge 3$, $r_2:= (p+1)/2$, we obtain by the Sobolev embedding that
    \begin{equation*}
        C^{-1}\norm{1-\sqrt{\ga_2/\ga_1}}_{H^1(\IB^d)} \le  \norm{V_1-V_2}_{L^{r_d}(\IB^d)} + |1-\sqrt{a_2/a_1}|,
    \end{equation*}
    which yields
    \begin{equation*}
        C^{-1}\norm{\sqrt{\ga}_1-\sqrt{\ga_2}}_{H^1(\IB^d)}  \le  \norm{V_1-V_2}_{L^{r_d}(\IB^d)} + |a_1-a_2|.
    \end{equation*}
    Now, from \eqref{e:cond_lemma_ga} one gets
    \begin{equation*}
        \norm{V_j}_{L^p(\IB^d)} \le C_K N.
    \end{equation*}
    We use this and interpolation to go from $r_d$ to $1$ in the RHS: 
    \begin{equation} \label{e:first_est_lem}
        C^{-1}\norm{\sqrt{\ga}_1-\sqrt{\ga_2}}_{H^1(\IB^d)} \le  \norm{V_1-V_2}_{L^1(\IB^d)}^{\beta_d} + |a_1-a_2|,
    \end{equation}
    where $\beta = \frac{p-r_d}{r_d(p-1)}$.

    We now want to show that $\Delta v \in L^r(\IB^d)$ for some $1<r<p$.   First we need to estimate the $vV_1$ term. If $1/q := 1/r-1/p$
    \begin{equation*}
        \norm{vV_1}_{L^r} \le \norm{V_1}_{L^p(\IB^d)} \norm{v}_{L^q(\IB^d)} \le C\norm{V_1}_{L^p(\IB^d)} \norm{v}_{H^1(\IB^d)},
    \end{equation*}
    by the Sobolev embedding, provided we can choose $1/q> 1/2-1/d$. This gives us the condition $ 1/r-1/p >1/2-1/d$, which is satisfied for any $1<r<p$ if $d=2$, or for $1<r \le 2d/(2+d)$ when $d\ge 3$. In particular we can choose again $r=r_d$, which yields
    \begin{equation*}
        \norm{\Delta (\sqrt{\ga}_1-\sqrt{\ga_2})}_{L^{r_d}(\IB^d)} \le \norm{V_1}_{L^p(\IB^d)}\norm{\sqrt{\ga}_1-\sqrt{\ga_2}}_{H^1(\IB^d)} + \sqrt{K} \norm{V_1-V_2}_{L^{r_d}(\IB^d)}.
    \end{equation*}
    so the same interpolation as before in the RHS gives
    \begin{equation*}
        C^{-1}\norm{\Delta (\sqrt{\ga}_1-\sqrt{\ga_2})}_{L^{r_d}(\IB^d)} \le \norm{\sqrt{\ga}_1-\sqrt{\ga_2}}_{H^1(\IB^d)} +  \norm{V_1-V_2}_{L^1(\IB^d)}^{\beta}.
    \end{equation*}
    Since $r_d>1$, by Calderón-Zygmund estimates we have
    \begin{equation*}
        \norm{ \sqrt{\ga}_1-\sqrt{\ga_2} -(\sqrt{a_1}-\sqrt{a_2})}_{W^{2,r_d}(\IB^d)} \le C\norm{\Delta (\sqrt{\ga}_1-\sqrt{\ga_2})}_{L^{r_d}(\IB^d)},
    \end{equation*}
    where we have subtracted the constant function $\sqrt{a_1}-\sqrt{a_2}$ to get a function with zero Dirichlet datum in the norm in the LHS.
    Putting the previous two estimates together, it follows that
    \begin{equation*}
        C^{-1} \norm{ \sqrt{\ga}_1-\sqrt{\ga_2}}_{W^{2,r_d}(\IB^d)} \le \norm{\sqrt{\ga}_1-\sqrt{\ga_2}}_{H^1(\IB^d)} +  \norm{V_1-V_2}_{L^1(\IB^d)}^\beta + |a_1-a_2|.
    \end{equation*}
    Combining this estimate with \eqref{e:first_est_lem} gives
    \begin{equation*}
        C^{-1} \left(  \norm{\sqrt{\ga}_1-\sqrt{\ga_2}}_{H^1(\IB^d)} +\norm{\sqrt{\ga}_1-\sqrt{\ga_2}}_{W^{2,1}(\IB^d)}  \right)  \le \norm{V_1-V_2}_{L^1(\IB^d)}^\beta + |a_1-a_2|.
    \end{equation*}
    To finish, we use that there exists a positive constant $C= C_{M,p,d}$ such that
    \begin{multline} \label{e:sqrt_removal}
       \norm{{\ga}_1-{\ga_2}}_{H^1(\IB^d)} +\norm{{\ga}_1-{\ga_2}}_{W^{2,1}(\IB^d)} \\ \le  C \left( \norm{\sqrt{\ga}_1-\sqrt{\ga_2}}_{H^1(\IB^d)} +\norm{\sqrt{\ga}_1-\sqrt{\ga_2}}_{W^{2,1}(\IB^d)} \right).
    \end{multline}
    This follows from the following two facts. First, $\ga_1 -\ga_2 = \psi (\sqrt{\ga}_1-\sqrt{\ga_2})$, where $\psi: = \sqrt{\ga}_1+\sqrt{\ga_2}$ belongs to $W^{2,p}(\IB^d)$. Second, if $U$ is an open bounded set, one can verify that multiplication by a function in $W^{2,p}(U)$ with $p>d/2$ is bounded in both $W^{2,1}(U)$ and $H^1(U)$.

     \textit{Local estimate:}  
     Let $I_s:= [s,1]$. 
     If we define the radial profiles $u_j(r)$ and $q_j(r)$ by the identities $q_j(|x|) = V_j(x)$ and $u_j(|x|) = \sqrt{\ga_j(x)}$, we have that
     \begin{equation*}
        -r^{-(d-1)}\p_r ( r^{d-1} \p_r u_j(r)) + q_j(r) u_j(r) = 0.
     \end{equation*}
     If $w := u_1-u_2$ it follows that
    \begin{equation*}
        (r^{d-1} w')'  = r^{d-1}g, \qquad g : = q_1 w + (q_1-q_2)u_2.
     \end{equation*}
    Therefore
     \begin{equation} \label{e:local_w}
         w'(r) = r^{1-d} \left(w'(1)- \int_r^1 t^{d-1} g(t) \, dt \right), \qquad w(r) = w(1) - \int_r^1 w'(t) \, dt.
     \end{equation}
     Thus
     \begin{equation*}
         \norm{w'}_{L^1(I_s)} \le C_{d,s}\left(|w'(1)|+\norm{g}_{L^1(I_s)}\right), \qquad \norm{w''}_{L^1(I_s)} \le C_{d,s}\left(|w'(1)|+\norm{g}_{L^1(I_s)}\right),
     \end{equation*}
     since $w''(r) = \frac{d-1}{r}  w'-g$.
     As a consequence, we have
     \begin{equation*}
         \norm{g}_{L^1(I_s)} \le\norm{w}_{L^\infty(0,1)}\norm{q_1}_{L^1(I_s)} + \sqrt{K} \norm{q_1-q_2}_{L^1(I_s)}.
     \end{equation*}
    It remains to estimate $\norm{w}_{L^\infty(I_s)}$, which can be achieved using Grönwall's inequality. Combining the identities in \eqref{e:local_w} and changing the order of integration, one obtains
    \begin{equation*}
        w(r) = w(1) + w'(1) h(r) + \int_r^1 J(r,t)(q_1(t) w(t) +(q_1(t)-q_2(t))u_2(t)) \, dt.
    \end{equation*}
    Since $h$ and $J$ are uniformly bounded, respectively, in $I_s$ and $I_s\times I_s$, one gets
    \begin{equation*}
        C^{-1}|w(r)| \le |w(1)| + |w'(1)|  + \norm{q_1-q_2}_{L^1(I_s)} +  \int_r^1 |w(t)||q_1(t)| \, dt.
    \end{equation*}
    By Grönwall's inequality, it follows that
    \begin{equation*}
        C^{-1}|w(r)| \le \left(|w(1)| + |w'(1)|  + \norm{q_1-q_2}_{L^1(I_s)}\right) \exp \left(\norm{q_1}_{L^1(I_s)} \right),
     \end{equation*}
     so that
     \begin{equation*}
        C^{-1}\norm{w}_{L^\infty} \le |w(1)| + |w'(1)|+ \norm{q_1-q_2}_{L^1(I_s)}.
     \end{equation*}
     Combining the previous estimates, we have shown that
      \begin{equation*}
        C^{-1}\left(  \norm{w}_{L^1(I_s)}+ \norm{w'}_{L^1(I_s)} + \norm{w''}_{L^1(I_s)} \right) \le |w(1)| + |w'(1)|+ \norm{q_1-q_2}_{L^1(I_s)}.
     \end{equation*}
     To finish, one can use the same argument as in \eqref{e:sqrt_removal}.

\bibliographystyle{myalpha}

\bibliography{references_radial_calderon}

\end{document}